\documentclass[11pt,a4paper]{article}

\newcommand{\stilde}{{\tilde{s}}}
\newcommand{\constant}{\mbox{\it\~c}}
\usepackage{xcolor}
\usepackage{a4wide}
\usepackage[utf8]{inputenc}
\usepackage[T1]{fontenc}
\usepackage[fleqn]{amsmath}
\usepackage{amssymb}
\usepackage{amsthm}
\usepackage{mathabx}
\usepackage{eucal} 
\usepackage[shortlabels]{enumitem}
\setenumerate{itemsep=0.005ex}
\usepackage{microtype}
\usepackage[longnamesfirst,round]{natbib}
\allowdisplaybreaks

\theoremstyle{plain}
\newtheorem{theorem}{Theorem}
\newtheorem{lemma}{Lemma}[section]

\newtheorem{externaltheorem}[lemma]{Theorem}
\newtheorem*{theoremC}{Theorem~\ref{5}}
\newtheorem*{theoremT}{Theorem~\ref{4}}
\newtheorem*{theoremD}{Theorem~\ref{2}}
\newtheorem*{theoremU}{Theorem~\ref{1}}
\newtheorem*{theoremCi}{Theorem~\ref{3}}

\theoremstyle{definition}
\newtheorem{definition}[lemma]{Definition}
\newtheorem*{notation}{Notation}

\theoremstyle{remark}
\newtheorem{remark}[lemma]{Remark}

\renewcommand{\phi}{\varphi}
\newcommand{\N}{{\mathbb{N}}}
\newcommand{\Q}{{\mathbb{Q}}}
\newcommand{\R}{{\mathbb{R}}}

\newcommand{\measure}[1]{\mu\!\left( #1 \right)}

\newcommand{\ceil}[1]{\lceil #1 \rceil } 
\newcommand{\bigceil}[1]{\left\lceil #1 \right\rceil }

\newcommand{\digits}[2]{\mathcal{L}(#1,#2)} 
\renewcommand{\L}{{\mathcal{L}}}

\renewcommand{\>}{\rangle}

\newcommand {\base}[2]{\langle{#1};{#2}\rangle}

\newcommand{\expa}[1]{\{#1\}}

\newcommand{\occ}{occ}

\newcommand{\card}{\mbox{\raisebox{.13em}{{$\scriptstyle \#$}}}}
\newcommand{\Tmax}{T_{\mbox{max}}}

\newcommand{\betterk}{k}
\newcommand{\betterN}{N}
\newcommand{\kbar}{{\overline{k}}}
\newcommand{\nextr}{r}

\newcommand{\bSigma}{{\mathbf\Sigma}}
\newcommand{\bPi}{{\mathbf\Pi}}

\begin{document}

\title{On the Normality of Numbers in Different Bases}

\author{
 \begin{tabular}{ccc}
Ver\'onica Becher&\ \  \ \ \ \ \  \ & Theodore A.~Slaman\\
\small Universidad de Buenos Aires&&\small  University of California Berkeley\\
\small vbecher@dc.uba.ar&&\small  slaman@math.berkeley.edu
  \end{tabular}
}

\date{March 22, 2013}

\maketitle

\section{Introduction}

We ask whether normality to one base is related to normality to another.  Maxfield in 1953 proved
that a real number is normal to a base exactly when it is normal to every base multiplicatively
dependent to that base (two numbers are multiplicatively dependent when one is a rational power of the
other.)  \cite{Sch61} showed that this is the only restriction on the set of bases to which a real
number can be normal. He proved that for any given set of bases, closed under multiplicative
dependence, there are real numbers normal to every base from the given set and not normal to any
base in its complement.  This result, however, does not settle the question of whether the
discrepancy functions for different bases for which a real number is normal are pairwise
independent.  Nor does it answer whether the set of bases for which a real number is normal plays a
distinguished role among its other arithmetical properties.

We pose these problems by means of mathematical logic and descriptive set theory.  The set of real
numbers that are normal to a least one base is located in the fourth level of the Borel hierarchy.
Similarly, the set of indices for computable real numbers that are normal to at least one base is located at
the fourth level of the arithmetic hierarchy.  In Theorem~\ref{1} we show that from both points of
view, the property that a real number is normal to at least one base is complete at the fourth level
($\bSigma_4^0$ and $\Sigma^0_4,$ respectively).  This result settles a question in \cite{Bug12} and
confirms a conjecture of A.~Ditzen~\cite[see][]{KiLin94}.  We obtain the result by first establishing
in Theorem~\ref{2}, that for any set at the third level of the arithmetic hierarchy ($\Pi^0_3$), there is a
computable real number which is normal exactly to the bases multiplicatively dependent to elements
of that set.  Theorem~\ref{3} exhibits a fixed point: for any property of bases expressed at the
third level of the arithmetic hierarchy ($\Pi^0_3$) and closed under multiplicative dependence, there is a real
number $\xi$ such that the bases which satisfy the property relative to $\xi$ are exactly those for
which $\xi$ is normal.

Theorem \ref{4} shows that the discrepancy functions for different bases can go to zero
independently.  We construct absolutely normal real numbers such that their discrepancy functions
for a given base $s$ converge to zero arbitrarily slowly and such that their discrepancies for all
the bases multiplicatively independent to $s$ are eventually dominated by a single computable bound.
In contrast, the real numbers constructed by \cite{Sch61} are not normal to a given  base $s$ and the discrepancy
functions for all bases multiplicatively independent to $s$ converge to zero at a prescribed rate.
With a different proof, \cite{BMP85} extended Schmidt's result and then \cite{MorPea88} gave explicit
bounds for the rate obtained with their method.
 In our construction the nonconforming behavior of the constructed real number 
 with respect to base $s$ appears even though it is normal to base $s$.

Theorem~\ref{5} sharpens Theorem~1 in \cite{Sch61}.  We construct  
a real number that is normal for all elements in a given set and denies even simple normality 
to all other elements, addressing an issue raised in \cite{BMP85}.

Normality is an almost-everywhere property of the real numbers: the set of normal numbers has
Lebesgue measure one.  Normality in some bases and not all of them is also an almost-everywhere
property, albeit not in the sense of Lebesgue. Consider the Cantor set $C_s$ obtained by omitting
the last digit (or two) in the base $s$ expansions of real numbers ($s$ greater than $2$).  Clearly, no element
of $C_s$ is simply normal to base $s$.  However, viewed from the perspective of the uniform measure
on this Cantor set, \cite{Sch60} shows that the subset of $C_s$ whose elements are normal to every
base $r$ multiplicatively independent to $s$ has measure one.

Our focus is on constructing real numbers and maintaining independent control over their discrepancy
functions for multiplicatively independent bases.  Since almost every element of $C_s$ is normal to
base $r$, almost every sufficiently long finite initial segment of a real in $C_s$ has small
discrepancy from normal in base $r$.  It is our task to convert this observation into methods of
constructing real numbers by iteratively extending their expansions in various bases.  The first
part of our task is to give computable bounds on discrepancy and estimates on how quickly
discrepancy for base $r$ decreases almost everywhere in $C_s$.  The second part is to convert these
finitary bounds into modules for constructions.  The typical module lowers discrepancy in bases $r$
from a finite set $R$ and increases discrepancy in a multiplicatively independent base $s$.  It is
important that the estimates on discrepancy be applicable in any basic open neighborhood in $C_s$ so
that the modules can be used as any finite point in the construction.

\section{Theorems}

\begin{notation}
  A~{\em base} is an integer greater than or equal to $2$.  For a real number $\xi$, we use
  $\expa{\xi}$ to denote its fractional part.  We write $\vec{\xi}$ to denote a sequence and $\xi_j$
  to denote the $j$th element of $\vec{\xi}$.  If~$\vec{\xi}$ is finite with $N$ elements we write
  it as $(\xi_1,\ldots, \xi_N)$.  For a subinterval $I$ of $[0,1]$, $\measure{I}$ is its measure,
  equivalently its length.  For a finite set $S$, $\card S$ is its cardinality.  We often drop the
  word {\em number} and just say {\em a real} or {\em a rational } or {\em an integer}.

\end{notation}

We recall the needed definitions and then state our results.
The usual presentation of the property of normality  to a given base 
for a real number is  in terms of counting occurrences of blocks of digits in its expansion in that base
(\cite{Bug12,KuiNie74}).  Absolute normality is normality to all bases.  
We define normality in terms of discrepancy.  See either of the above references for a proof of Wall's Theorem, which
establishes the equivalence.

\begin{definition}\label{2.1}
The \emph{discrepancy} of a sequence $\vec{\xi}=(\xi_1,\dots, \xi_N)$ of real numbers in the unit
interval is
  \[
  D(\vec{\xi})=\sup_{0\leq u< v\leq 1}
  \Bigl|
  \frac{\card\{n:1 \leq n \leq N, u \leq \xi_n < v\}}{N}-(v-u)
  \Bigr|.
  \]
\end{definition}

When we refer to a sequence by specifying its elements, we will write $D(\xi_1,\dots, \xi_N)$,
rather than $D((\xi_1,\dots, \xi_N))$.

\begin{definition}\label{2.2}
  Let $r$ be a base.  A real number $\xi$ is \emph{normal to base $r$} if and only if
  $\lim_{N\to\infty}D(\expa{r^j\xi}:0\leq j< N)=0$.  \emph{Absolute normality}  is
  normality to every base.
\end{definition}

  A formula in the language of arithmetic 
 is $\Pi^0_0$ and $\Sigma^0_0$ if
  all of its quantifiers are bounded.  It is $\Sigma^0_{n+1}$ if it has the form $\exists x\, \theta$
  where $\theta$ is $\Pi^0_n$ and it is $\Pi^0_{n+1}$ if it has the form $\forall x\, \theta$ where
  $\theta$ is $\Sigma^0_n$.  A subset $A$ of $\N$ is $\Sigma^0_n$ (respectively, $\Pi^0_n$) if there
  is a $\Sigma^0_n$ (respectively, $\Pi^0_n$) formula $\phi$ such that for all $n$, $n\in A$ if and
  only if $\phi(n)$ is true.
  A $\Sigma^0_n$  subset $A$ of the natural numbers is $\Sigma^0_n$-complete if there is a computable
  function $f$ mapping $\Sigma^0_4$ formulas to natural numbers such that for all $\phi$, $\phi$ is
  true in the natural numbers if and only if $f(\phi)\in A.$

  The Borel hierarchy for subsets of $\R$ with the usual topology 
  states that a set $A$ is $\bSigma^0_1$ if and only if $A$  is open
  and $A$ is $\bPi^0_1$ if and only if $A$ is closed.  
  $A$ is $\bSigma^0_{n+1}$ if and only if it is a
  countable union of $\bPi^0_{n}$ sets and  $A$  is $\bPi^0_{n+1}$ if and only if it is a countable
  intersection of $\mathbf\Sigma^0_n$ sets.  
  By an important theorem, a $\mathbf\Sigma^0_n$ subset of $\R$ is $\mathbf\Sigma^0_n$-complete if
  and only if it is not $\mathbf\Pi^0_n$.

\begin{theorem}\label{1}
(1) The set  of indices for computable real numbers which are normal at least one base is $\Sigma^0_4$-complete. 
(2) The set of real numbers that are normal to at least one base is $\mathbf \Sigma^0_4$-complete.
\end{theorem}

\begin{remark}
A routine extension of the proof shows that the set of real numbers which are normal to
infinitely many bases is $\bPi^0_5$-complete.  Expressed in terms of the complement, the set of real numbers
which are normal to only finitely many bases is $\bSigma^0_5$-complete.
\end{remark}

Let $M$ be the set of minimal representatives of the multiplicative dependence equivalence
classes. Our proof of Theorem~\ref{1} relies on the following.

\begin{theorem}\label{2}
  For any $\Pi^0_3$ subset $R$ of $M$ there is a computable real number $\xi$ such that for all $r$
  in $M$, $r\in R$ if and only if $\xi$ is normal to base $r$.  Furthermore, $\xi$ is computable
  uniformly in the $\Pi^0_3$ formula that defines $R$.
\end{theorem}

Theorem~\ref{3} exhibits a fixed point: the real $\xi$ appears in the $\Pi^0_3$ definition of its
input set.  It asserts that the set of bases for which $\xi$ is normal can coincide with any other
property of elements of $M$ definable by a $\Pi^0_3$ formula relative to $\xi$.  Thus, the set of
bases for normality can be arbitrary, nothing distinguishes it from other $\Pi^0_3$ predicates on
$M$.  As a subset of $\N$ its only distinguishing feature is that it is closed under multiplicative
dependence.

\begin{theorem}\label{3}
For any $\Pi^0_3$ formula $\phi$ there is a computable real number $\xi$ such that for any base $r\in M$, 
$\phi(\xi,r)$ is true if and only if  $\xi$ is normal to base $r$.
\end{theorem}

Theorem~\ref{4} illustrates the independence between the discrepancy functions for multiplicatively
independent bases by exhibiting an extreme case, that all but one of the bases behave predictably
and the other is arbitrarily slow.

\begin{theorem}\label{4}
  Fix a base $s$.  There is a computable function $f:\N\to\Q$ monotonically decreasing to
  $0$ such that for any function $g:\N\to\Q$ monotonically decreasing to $0$ there is an
  absolutely normal real number $\xi$ whose discrepancy for base $s$ eventually dominates $g$ and
  whose discrepancy for each base multiplicatively independent to $s$ is eventually dominated by $f$.
  Furthermore, $\xi$ is computable from $g$.
\end{theorem}

\begin{remark}
The proof  Theorem~\ref{4} can be adapted produce other contrasts 
in behavior between multiplicatively independent bases.  We give two  examples.

(1) Let $s$ be a base.  There is a computable function $f:\N\to\Q$ monotonically decreasing to $0$ such
that for any function $g:\N\to\N$, there is an absolutely normal real number $\xi$ such that its
discrepancy for $s$ satisfies for all $n$ there is an $N>g(n)$ such that 
$D(\expa{s^j\xi}:0\leq j< N)>1/n$ and its discrepancies for bases multiplicatively independent to $s$ are eventually
bounded by $f$.  Furthermore, $\xi$ is computable from any real number $\rho$ which can computably
approximate $g$.

(2) Let $s$ and $r$ be multiplicatively independent bases. There is a computable absolutely normal
number $\xi$ such that
\[
\limsup_{N\to\infty}\frac{D(\expa{s^j\xi}:0\leq j< N)}{D(\expa{r^j\xi}:0\leq j< N)} =
\limsup_{N\to\infty}\frac{D(\expa{r^j\xi}:0\leq j< N)}{D(\expa{s^j\xi}:0\leq j< N)}
=\infty.
\]
\end{remark}
\begin{remark}\label{2.5}
  There is a computable function $f:\N\to\Q$ monotonically decreasing to $0$ such that the
  discrepancy of almost every real number  is eventually dominated by $f$.  In contrast,
  there is no computable function which dominates the discrepancy of all the computable
  absolutely normal numbers.
\end{remark}

Finally, we state the improvement of Theorem 1 of \cite{Sch61}, asserting 
simple normality in the conclusion.

\begin{definition}\label{2.6}
  Let $N$ be a positive integer. Let $\xi_1,\dots, \xi_N$ be real numbers in $[0,1]$.  Let $F$ be a
  family of subintervals.  The discrepancy of $\xi_1,\dots,\xi_N$ for $F$ is
\[
  D(F,(\xi_1,\dots,\xi_N))=\sup_{I\in F}
  \Bigl|
  \frac{{\card}\{n:\xi_n \in I\}}{N}-\measure{I}
  \Bigr|.
  \]
\end{definition}

\begin{definition}\label{2.7}
  Let $r$ be a base and let $\xi$ be a real number.  Let $F$ be the set of intervals of the form
  $[a/r,(a+1)/r)$, where $a$ is an integer $0\leq a<r$.  $\xi$ is \emph{simply normal to base $r$} if
  $\lim_{N\to\infty}D(F,(\expa{r^j\xi}:0\leq j< N))=0.$
\end{definition}

\begin{theorem}\label{5}
  Let $R$ be a set of bases closed under multiplicative dependence.  There are real numbers normal
  to every base from $R$ and not simply normal to  any base in its complement.  Furthermore, such a real number
  can be obtained computably from $R$.
\end{theorem}

\section{Lemmas}

\subsection{On Uniform Distribution of Sequences}

\begin{lemma}\label{3.1}
  Let  $\epsilon$ be a real number strictly between $0$ and $1$.  Let $F_\epsilon$ be the
  family of semi-open intervals $B_a=[a/\ceil{3/\epsilon}, (a+1)/\ceil{3/\epsilon})$, where $a$ is
  an integer $0\leq a<\ceil{3/\epsilon}$. For any sequence $\vec{\xi}$ and any $N$, if
  $D(F_\epsilon,\vec{\xi})<(\epsilon/3)^2$ then $D(\vec{\xi})<\epsilon$.
\end{lemma}

\begin{proof}
  Let $\vec{\xi}$ be a sequence of real numbers of length $N$ such that $D(F_\epsilon,\vec{\xi})$ is
  less than $(\epsilon/3)^2$.  Let $I$ be any semi-open subinterval of $[0,1]$.  Denote
  $\ceil{3/\epsilon}$ by $n$.
  The number of $B_a$ with nonempty intersection with $I$ 
  is less than or equal to $\ceil{n\measure{I}}$.   For each $B_a\in F_\epsilon$, $\card\{\xi_n:\xi_n\in
  B_a\}$ is less than or equal to $(1/n+\epsilon^2/9)N$.  Thus, by the definition of $n$,
  \begin{align*}
    \frac1N\card\{\xi_n:\xi_n\in I\}&\leq \frac1N\ceil{n\measure{I}} (1/n+\epsilon^2/9)N \leq \measure{I}+\epsilon.
   \end{align*}
   Similarly, $\frac1N\card\{\xi_n:\xi_n\in I\}\geq \measure{I}-\epsilon.$
\end{proof}

\begin{remark}\label{3.2}
In Lemma~\ref{3.1},  $F_\epsilon$ can be replaced by any partition of $[0,1]$ 
into subintervals of equal length, each of length at most $\epsilon/3$.
\end{remark}

We record the next three observations without proof.

\begin{lemma}\label{3.3}
  Suppose that $\epsilon$ is a positive real, $\vec{\xi}$ is a sequence of length $N$ and that
  $D(\vec{\xi})<\epsilon$.  For any sequence $\vec{\nu}$ of length $n$ with $n<\epsilon N$, for all
  $k\leq n$, $D(\nu_1\dots, \nu_k, \xi_1,\dots, \xi_N)<2\epsilon$ and $D(\xi_1,\dots,
  \xi_N,\nu_1\dots, \nu_k)<2\epsilon$.
\end{lemma}

%

\begin{lemma}\label{3.4}
  Let $\vec{\xi}$ be a sequence of real numbers, $\epsilon$ a positive real and $(b_m:0\leq
  m<\infty)$ an increasing sequence of positive integers.  Suppose that there is an $m_0$ such that
  for all $m>m_0$, $b_{m+1}-b_m\leq \epsilon b_m$ and $D(\xi_j:b_m< j\leq
  b_{m+1})<\epsilon$.  Then $\lim_{N\to\infty}D(\vec{\xi})\leq 2\epsilon.$
\end{lemma}


\begin{lemma}\label{3.5}    
  Let $m$ be a positive integer and $I$ a semi-open interval.  Suppose $\vec{\xi}$ is a sequence
  of real numbers of length $N$ such that $N\geq\ceil{2m/\measure{I}}$ and  for all $j$ with 
$m\leq j\leq N$, $\xi_j\not\in  I$.  Then, $D(I,\vec{\xi})\geq\mu(I)/2$.
\end{lemma}

\begin{notation}
  We let $e(x)$ denote $e^{2\pi i x}$.
\end{notation}

\begin{externaltheorem}[Weyl's Criterion \protect{\cite[see][]{Bug12}}]
A sequence $(\xi_n:n\geq 1)$ of real numbers is uniformly distributed modulo one if and only if for
every non-zero $t,$ $\displaystyle{\lim_{N\to\infty}\frac1N\sum_{j=1}^N e(t\xi_n)=0.}$
\end{externaltheorem}

\begin{externaltheorem}[\protect{\textbf{LeVeque's Inequality} \cite[see][Theorem~2.4]{KuiNie74}}]\label{3.7}
Let $\vec{\xi}=(\xi_1,\dots,\xi_N)$  be  a finite sequence. Then,
$\displaystyle{
  D(\vec{\xi})\leq \Bigl(\frac6{\pi^2}\;\sum_{h=1}^\infty \frac1{h^2} \Bigl|\frac1N\; \sum_{j=1}^N
  e(h\xi_j)\Bigr|^2\Bigr)^{\frac13}.}$
\end{externaltheorem}

\begin{lemma}\label{3.8}
  For any positive real $\epsilon$ there is a finite set $T$ of integers  and a positive real~$\delta$ 
  such that for any $\vec{\xi}=(\xi_1,\dots,\xi_N)$, if for all $t\in T$,
  $\displaystyle{\frac1{N^2}\; \Bigl|\sum_{j=1}^N e(t\xi_j)\Bigr|^2<\delta}$ then
  $D(\vec{\xi})<\epsilon.$ Furthermore, such $T$ and $\delta$ can be computed from $\epsilon$.
\end{lemma}

\begin{proof}
By LeVeque's Inequality, $\displaystyle{
  D(\vec{\xi})\leq \Bigl(\frac6{\pi^2}\;\sum_{h=1}^\infty \frac1{h^2} \Bigl|\frac1N\; \sum_{j=1}^N
  e(h\xi_j)\Bigr|^2\Bigr)^{\frac13}.}$
  Note that\linebreak
  $  \Bigl|\frac1N\; \sum_{j=1}^N
  e(h\xi_j)\Bigr|^2\leq 1.$
Hence, for each $h$,
\begin{align*}
\sum_{h=m+1}^\infty \frac1{h^2} \Bigl|\frac1N\; \sum_{j=1}^N
e(h\xi_j)\Bigr|^2&\leq \sum_{h=m+1}^\infty \frac1{h^2} 
\leq \int_{m+1}^\infty x^{-2} dx
\leq \frac1{m+1}.
\end{align*}
Assume $\displaystyle{\frac1{N^2}\; \Bigl|\sum_{j=1}^N e(t\xi_j)\Bigr|^2<\delta}$ for all
positive integers $t$ less than or equal to $m$. Then,
\begin{align*}
\sum_{h=1}^m \frac1{h^2} \Bigl|\frac1N\; \sum_{j=1}^N
e(h\xi_j)\Bigr|^2
+\sum_{h=m+1}^\infty \frac1{h^2} \Bigl|\frac1N\; \sum_{j=1}^N
e(h\xi_j)\Bigr|^2\leq
\sum_{h=1}^m \frac{1}{h^2} \delta
+\frac1{m+1} 
\leq\delta m+\frac1{m+1}.
\end{align*}
To ensure $D(\vec{\xi})<\epsilon$ it is sufficient that 
$(6/\pi^2)\bigl(\delta m+(1/m+1)\bigr)^{\frac13}<\epsilon$.
This is obtained by setting
$\delta m< (1/2) ( \epsilon^3\pi^2/ 6)$ and 
$1/(m+1)< (1/2) ( \epsilon^3\pi^2/ 6)$.
Let  $m = \ceil{12/(\epsilon^3\pi^2)}$,
$T=\{ 1, 2, \ldots, m\}$  and 
$\delta= (\epsilon^3\pi^2) / (24 m)$.
\end{proof}

\subsection{On Normal Numbers}

\begin{notation}
  We use $\base{b}{r}$ to denote $\ceil{b/\log r}$, where $\log$ refers to natural logarithm.  We
  say that a rational number $\eta\in[0,1]$ is $s$-adic when $\eta=\sum_{j=1}^ad_js^{-j}$ for digits
  $d_j$ in $\{0,\dots,s-1\}$.  In this case, we say that $\eta$ has precision $a$.  We use
  $\digits{s}{k}$ to denote sequences in the alphabet $\{0,\dots,s-1\}$ of length $k$.  For a
  sequence $w$, we write $|w|$ to denote its length.  When $1\leq i\leq j\leq |w|$, we call
  $(w_i,\dots,w_j)$ a block of $w$.  The number of occurrences of the block $u$ in $w$ is
  $\occ(w,u)=\#\{ i: (w_i,\dots,w_{i+|u|-1})=u \}$.
\end{notation}

\begin{lemma}\label{3.9}
  Let $s$ and $r$ be bases, $a$ be a positive integer and $\epsilon$ be a real between $0$ and~$1$.
  There is a finite set of intervals $F$ and a positive integer $\ell_0$ such that for all $\ell\geq\ell_0$
  and all $\xi_0$, if $\xi\in[\xi_0,\xi_0+s^{-\base{a+\ell}{s}})$ and $D(F,(\expa{r^j \xi_0}:\base{a}{r} <  j\leq \base{a+\ell}{r}))<(\epsilon/10)^4$ then
  $D(\expa{r^j \xi}:\base{a}{r}< j\leq \base{a+\ell}{r})<\epsilon$.
  Furthermore, $\ell_0$ and $F$ can be taken as computable functions of $r$ and $\epsilon$.
\end{lemma}

\begin{proof}
  Let $F_\epsilon$ be as in Lemma~\ref{3.1} and let $I$ be an interval in $F_\epsilon$.  Let $n$
  denote $\ceil{100/\epsilon^2}.$ Let $F$ be the set of semi-open intervals $B_c=[c/n, (c+1)/n)$,
  where $0\leq c<n$.  For the sake of computing $\ell_0$, consider $b>a$, $\xi$ and $\xi_0$ such
  that $\xi\in[\xi_0,\xi_0+s^{-\base{b}{s}})$.  Assume $D(F,(\expa{r^j
    \xi_0}:\base{a}{r}< j\leq \base{b}{r}))<(\epsilon/10)^4$.

  Note that for all $j$ less than $\base{b}{r}-\log n/\log r-1$, we have $r^js^{-\base{b}{s}}< 1/n$. 
   Hence,
  for all but the last $\ceil{\log n/\log r}+2$ values of $j$, $|r^j\xi_0 - r^j\xi|<1/n$.
  Let $C$ be the set of intervals $B_c$ such that either $B_c$ or $B_{c+1}$ has non-empty
  intersection with $I$.  If $j$ is less than $\base{b}{r}-\log n/\log r-1$ then 
  $\expa{r^j\xi}\in I$ implies that $\expa{r^j\xi_0}\in \cup C$.  
  Observe that $\#C\leq \ceil{n \measure{I}}+2$.  The fraction
\[
\frac1{\base{b}{r}-\base{a}{r}}
\#\{j : \base{a}{r}< j <\base{b}{r} -\log n/\log r-1  \text{ and } \expa{r^j\xi}\in \cup C\} 
\]
is at most $\ceil{n\measure{I}+2} (1/n+\epsilon^4/10^4)$.
And by definition of $n$,
$$  
  (n\measure{I}+3) (1/n+(\epsilon/10)^4)   \leq
 \measure{I}+\ceil{100/\epsilon^2}(\epsilon/10)^4 +
    3/\ceil{100/\epsilon^2}+3(\epsilon/10)^4  
\leq 
  \measure{I}+ (1/2)(\epsilon/3)^2.
$$
There are at most $\ceil{\log n/\log r}+2$ remaining  $j$, 
those for which $j\geq \base{b}{r}-\log n/\log r-1$.
Suppose that for each such $j$,  $r^j\xi\in I$.  Then,
   \begin{align*}
\frac{\ceil{\log n/\log r} + 2}{\base{b}{r}-\base{a}{r}} &\leq
\frac{\log n/\log r + 3}{\base{b}{r}-\base{a}{r}} 
\leq
\frac{\log \ceil{100/\epsilon^2} +3\log r}{b-a-\log r}.
   \end{align*}
   Let $\ell_0$ be $\bigceil{\log r+\frac{18}{\epsilon^2}\ceil{\log \ceil{100/\epsilon^2} +3\log
       r}}$.  For $b\geq a+\ell_0$, $\frac{\log \ceil{100/\epsilon^2} +3\log r}{b-a-\log
     r}<(1/2)(\epsilon/3)^2$.  A similar argument yields the same estimates for the needed lower
   bound.  Then, for $F_\epsilon$ and any $b\geq a+\ell_0$, $\displaystyle{
     D(F_\epsilon,(\expa{r^j\xi}:\base{a}{r}< j\leq\<b;r\>))<(\epsilon/3)^2.
   }$
   By applying Lemma~\ref{3.1}, for any $\ell\geq \ell_0$,
   $D(\expa{r^j \xi}:\base{a}{r}< j\leq
   \base{a+\ell}{r})<\epsilon$.
\end{proof}

\begin{definition}
  Fix a base $s$.  The \emph{discrete discrepancy} of $w\in\digits{s}{N}$ for a block of size $\ell$ is
  \[
  C(\ell,w)=\max\left\{\left|\frac{occ(w,u)}{N}-\frac{1}{s^\ell}\right|:u\in\digits{s}{\ell}\right\}.
\]
\end{definition}

The next lemmas relate the discrete discrepancy of sequences in $w\in\digits{s}{N}$ to the discrepancy of their
associated sequences of real numbers.

\begin{lemma}\label{3.11}
  Let $\epsilon$ be a positive real, $s$ a base, $\ell$ and $N$ positive integers such that
  $s^\ell>3/\epsilon$ and $N>2\ell(3/\epsilon)^2$, and $w\in\digits{s}{N}$ such that
  $C(\ell,w)<\epsilon^2/18$.  Then, $D(\expa{s^j\eta_w}:0\leq j<N)<\epsilon$, where
  $\eta_w=\sum_{j=1}^{|w|} w_j s^{-j}$.
\end{lemma}

\begin{proof}
  Let $\ell$ be such that $s^{-\ell}<\epsilon/3$.  Let $F$ be the set of $s$-adic intervals of
  length $s^{-\ell}$.  Any $I$ in $F$ has the form $[\eta_u,\eta_u+s^{-\ell})$, for some
  $u\in\digits{s}{\ell}$, and further, $\expa{s^j\eta_w}\in I$ if and only if the block $u$ occurs
  in $w$ at position $j+1$.  Thus, we can count instances of $\expa{s^j\eta_w}\in I$ by counting
  instances of $u$ in $w$.  Let $N$ and $w$ be given so that $N>2\ell(3/\epsilon)^2$,
  $w\in\digits{s}{N}$ and $C(\ell,w)<(\epsilon/3)^2/2$.  Then, for any $u\in\digits{s}{\ell}$,
  $\left| occ(w,u)/N-s^{-\ell} \right|<(\epsilon/3)^2/2$.  For any $I\in F$, $\displaystyle{
    \frac1N\card\left\{j:\expa{s^j\eta_w}\in I \mbox{ and } 0\leq j<N\right\}
    <s^{-\ell}+(\epsilon/3)^2/2+(\ell-1)/N< s^{-\ell}+(\epsilon/3)^2.}$ A similar count gives the
  analogous lower bound.  Hence, $D(F,(\expa{s^j\eta_w}:0\leq j<N))<(\epsilon/3)^2$ and so
  $D(\expa{s^j\eta_w}:0\leq j<N)<\epsilon$, by application of Lemma~\ref{3.1} and Remark~\ref{3.2}.
\end{proof}

\begin{lemma}[see Theorem~148, \protect{\cite{hardy}}]\label{3.12}
  For any base $s$, for any positive integer $\ell$ and for any positive real numbers $\epsilon$ and $\delta$, 
  there is an $N_0$ such that for all $N\geq N_0$,
  \[
    {\card}\Bigl\{v\in \digits{s}{N}: C(\ell,v)\geq\epsilon\Bigr\} < \delta s^N.   
  \]
  Furthermore, $N_0$ is a computable function of $s$, $\epsilon$ and $\delta$.
\end{lemma}

The next lemma is specific to base $2$ and will be applied in the proof of Theorem~\ref{5}.

\begin{lemma}\label{3.13}
  Given a positive real number $\epsilon$, there is an $N_0$ such that for all $N\geq N_0$,
  \[
  {\card}\Bigl\{v\in \digits{2}{N}:
  \frac{1}{2N}\card\bigl\{m:\expa{2^m\eta_v}\in[0,1/2)\bigr\}\geq 5/8
\Bigr\}> (1-\epsilon)2^{N}
\]
where for $v=(v_1,\dots,v_N)\in \digits{2}{N}$, $\eta_v=\sum_{j=1}^N v_j4^{-j}$. 
Furthermore, $N_0$ is a computable function of $\epsilon$.
\end{lemma}

\begin{proof}
By Lemma~\ref{3.12}, for any positive $\delta$ there is an $N_0$ such that for all $N\geq N_0$,
\[
{\card}\Bigl\{v\in \digits{2}{N}: C(1,v)\leq\delta\Bigr\} \geq (1-\epsilon) 2^{N}.
\]
Thus, for $(1-\epsilon)2^N$ many $v$,
$\displaystyle{
\left|\frac{\card\{n:v_n=0\}}{N}-\frac12\right|<\delta \quad\mbox{and}\quad
\left|\frac{\card\{n:v_n=1\}}{N}-\frac12\right|<\delta.
}$
Consider the natural bijection $V$ between $\digits{2}{N}$ and the set $\L$ of sequences
of length $N$ of symbols from $\{(00),(01)\}$.  Then for $(1-\epsilon)2^N$ many $v\in\digits{2}{N}$, 
\[
\left|\frac{\card\{n:V(v)_n=(00)\}}{N}-\frac12\right|<\delta
\quad\mbox{and}\quad \left|\frac{\card\{n:V(v)_n=(01)\}}{N}-\frac12\right|<\delta.
\]
We can construe each length $N$ sequence $V(v)$ from $\L$ as a length $2N$ binary
sequence $V^*(v)$.  Under this identification, $\eta_v=\sum_{j=1}^{2N}V^*(v)_j2^{-j}=\sum_{j=1}^N
v_j4^{-j}.$  For any $v\in\digits{2}{N}$,
\[
 \card\{m:V^*(v)_m=0\}=2\card\{n:V(v)_n=(00)\}+\card\{n:V(v)_n=(01)\}.
\]
So, for $(1-\epsilon)2^N$ many $v\in\digits{2}{N}$, 
\[
\card\{m:V^*(v)_m=0\}\geq 2 (1/2-\delta)N+(1/2-\delta)N=3/2N-3\delta N.
\]
Thus, $\card\{m:\expa{2^m\eta_v}\in[0,1/2) \mbox{ and }0\leq m< 2N\}\geq  (3/2)N-3\delta N$.
Hence, 
\[
\frac{1}{2N}\card\{m:\expa{2^m\eta_v}\in[0,1/2)\}\geq 3/4-3\delta/2.
\]
For $\delta=1/12$, the lemma follows.
\end{proof}

\begin{lemma}\label{3.14}
  Let $\epsilon$ be a positive real and let $s$ be a base.  There is a $k_0$ such that for every
  $k\geq k_0$ there is an $N_0$ such that for all $N\geq N_0$,
  \[
  \card\Bigl\{w\in \digits{\stilde}{N}:
     D(\expa{s^j\eta_w}: 0\leq j< kN)<\epsilon\Bigr\}
     > (1/2) \stilde^{\, N},
     \]
where $\stilde$ is either of $s^k-1$ or $s^k-2$, and for $w=(w_1,\dots,w_N)\in \digits{\stilde}{N}$,
$\eta_w=\sum_{j=1}^N w_j(s^{k})^{-j}$.  Furthermore, $k_0$ is a
computable function of $s$ and $\epsilon$ and $N_0$ is a computable function of $s$, $\epsilon$ and
$k$.
\end{lemma}

\begin{proof}
  Fix the real $\epsilon$ (to be used only at the end of the proof) and fix the base $s$.
  By Lemma~\ref{3.12}, for each real $\delta>0$ and integer $\ell>0$ there is
  $k_0$ such that $\ell/k_0<\delta$ and for all $k\geq k_0$
  \[
    \card\Bigl\{v\in \digits{s}{k}: C(\ell,v)<\delta \Bigr\}>(1-\delta)s^k.
    \]
    Consider such a $k$.  The elements $v\in\digits{s}{k}$ are of two types: those good-for-$\ell$
    with $C(\ell,v)<\delta$ and the others.  By choice of $k$, $(1-\delta)s^k$ blocks of length
    $k$ are good-for-$\ell$.  Let $\stilde$ be either $s^k-1$ or $s^k-2$.  Now view
    $\digits{\stilde}{1}$ in base $s$.  If $\stilde$ is $s^k-1$, then $\digits{\stilde}{1}$ lacks
    the not-good-for-$\ell$ block of $k$ digits all equal to $s-1$.  If $\stilde$ is $s^k-2$, then
    $\digits{\stilde}{1}$ also lacks the not-good-for-$\ell$ block of $k-1$ digits equal to $s-1$ followed by
    the final digit $s-2$.  So, at least $(1-\delta)$ of the elements in $\digits{\stilde}{1}$ are
    good-for-$\ell$ in that they correspond to good blocks of length $k$.
    Let $N_0$ be such that for all $N\geq N_0$,
  \[
    \card\Bigl\{w\in \digits{\stilde}{\ N}: C(1,w)<\delta \Bigr\}>(1-\delta)\stilde^{\, N}.
  \]
  Take $N\geq N_0$ and consider a sequence $w$ in $\digits{\stilde}{N}$.  If $C(1,w)<\delta$,
  then each element in $\digits{\stilde}{1}$ occurs in $w$ at least $N(1/\stilde-\delta)$ times.  
  Let $w\mapsto w^*$ denote the map that takes
  $w\in\digits{\stilde}{N}$ to $w^*\in\digits{s}{kN}$ such that $\displaystyle{
      \sum_{n=1}^{N}w_{n}(s^k)^{-n}=\sum_{n=1}^{kN} w^*_{n+1}\,s^{-n}}$.  
Let  $u\in\digits{s}{\ell}$. 
We obtain the following bounds for $\occ(u,w^*)$:
 \begin{align*}
 \occ(u,w^*)
 \leq\ & N(1/s^{\ell}+\delta)k+2\ell N+\delta Nk
 \\
 \leq\ &Nk(1/s^{\ell}+2\delta+2\ell/k).
\\
    \occ(u,w^*)
    \geq\ &  \sum_{i=0}^{N-1}\occ(u,(w^*_{ik+1},\dots,w^*_{ik+k}))&
\\
     \geq\    &  \stilde(1-\delta)   N(1/\stilde - \delta)  k (1/s^\ell-\delta)
       = Nk  (1-\delta) (1 - \stilde  \delta)     (1/s^\ell-\delta)
\\
      \geq\    &  Nk (1/s^\ell -\delta - s^k\delta/s^\ell-\stilde\delta^3) 
\\
      \geq\ & Nk(1/s^\ell-\delta s^k). \quad\text{(We can assume that $\delta<1/2$.)}
\end{align*}
 So $C(\ell,w^*)< \delta s^k$.  Hence,
    $    \card\Bigl\{w\in \digits{\stilde}{N}: C(\ell, w^*)< \delta s^k \Bigr\} \geq
    (1-\delta) \stilde^{\, N}.$
    Let $\delta= s^{-k}(\epsilon^2/18)$.  Then, 
       \begin{flalign*}
      \card\Bigl\{w\in \digits{\stilde}{N}: C(\ell, w^*)< \epsilon^2/18\}) \Bigr\}
      \geq & (1- s^{-k}(\epsilon^2/18))\stilde^{\, N}.
      \end{flalign*}
      In particular, this inequality holds for the minimal $\ell$ satisfying $s^\ell>3/\epsilon.$
      Since, $\epsilon$ can be chosen so that $(1- s^{-k}(\epsilon^2/18))$ is at least $1/2$, we can
      apply Lemma~\ref{3.11} to conclude the wanted result:
      $\displaystyle{
      \card\Bigl\{w\in \digits{\stilde}{N}: D(\expa{s^j\eta_w}: 0\leq j< kN)<\epsilon\Bigr\} >
      (1/2) \stilde^{\, N}.}$      
\end{proof}

\subsection{Schmidt's Lemmas}

Lemma~\ref{3.17}  is our  analytic tool to control discrepancy for multiplicatively independent
bases.  It  originates in \cite{Sch61}.  Our proof adapts the version  given in \cite{Pol81}.  

\addtocounter{footnote}{1}

\begin{lemma}[Hilfssatz~5, \cite{Sch61}]\label{3.15}
  Suppose that $r$ and $s$ are multiplicative independent bases. 
  There is a constant $c$, with $0<c<1/2$, depending only on $r$ and $s,$ such that for all natural
  numbers $K$ and $\ell$ with $\ell\geq s^K$,
\[
  \sum_{r=0}^{N-1}\prod_{k=K+1}^\infty |\cos(\pi r^n\ell/s^k)|\leq 2 N^{1-c}.    
  \]
  Furthermore, $c$ is a computable function of $r$ and $s$.\footnote{Actually, Schmidt asserts 
the computability of $c$ in separate paragraph (page 309 in the same article): 
``Wir stellen zun\"achst fest,
da\ss man mit etwas mehr M\"uhe Konstanten $a_{20}(r, s)$ aus Hilfssatz~5 explizit berechnen
k\"onnte, und da\ss\ dann $\xi$ eine eindeutig definierte Zahl ist.''}
\end{lemma}

\begin{definition}\label{3.16}
$ \displaystyle{ A(\xi,R,T,a,\ell)=\sum_{t\in T}\;\sum_{r\in R}\;
  \Bigl|\sum_{j=\base{a}{r}+1}^{\base{a+\ell}{r}} e(r^j t \xi)\Bigr|^2.}$
\end{definition}

\begin{lemma}\label{3.17}
  Let $R$ be a finite set of bases, $T$ be a finite set of non-zero integers and $a$ be a
  non-negative integer.  Let $s$ be a base multiplicatively independent to the elements of $R$ and
  let $c(R,s)$ be the minimum of the constants $c$ in Lemma~\ref{3.15} for pairs $r,s$ with $r\in
  R$.  Let $\stilde$ be $s-1$ if $s$ is odd and be $s-2$ if $s$ is even.  Let $\eta$ be $s$-adic
  with precision $\base{a}{s}$.  For $v\in \digits{\stilde}{N}$ let $\eta_v$ denote the rational
  number $\eta+s^{-\base{a}{s}}\sum_{j=1}^N v_js^{-j}$.  There is a length $\ell_0$ such that for
  all $\ell\geq \ell_0$, there are at least $(1/2) \stilde^{\base{a+\ell}{s}-\base{a}{s}}$ numbers
  $\eta_v$ such that $A(\eta_v,R,T,a,\ell)\leq \ell\,^{2-c(R,s)/4}$.  Furthermore, $\ell_0$ is a
  computable function of $R$, $T$ and $s$.
\end{lemma}

\begin{proof}
  We abbreviate $A(x,R,T,a,\ell)$ by $A(x)$, abbreviate $(a+\ell)$ by $b$ 
 and $\digits{\stilde}{\base{b}{s}-\base{a}{s}}$ by~$\L$.  
To provide the needed $\ell_0$ we will estimate  the mean value of $A(x)$ on the set of  numbers~$\eta_v$.
We need an upper bound for  
\[
    \sum_{v\in \L}A(\eta_v)=\sum_{v\in \L}\;\sum_{t\in T}\;\sum_{r\in R}\;
    \Bigl|\sum_{j=\base{a}{r}+1}^{\base{b}{r}} e(r^j t \eta_v)\,\Bigr|^2
=\sum_{v\in \L}\;\sum_{t\in T}\;\sum_{r\in R}\;
    \sum_{g=\base{a}{r}+1}^{\base{b}{r}}
    \sum_{j=\base{a}{r}+1}^{\base{b}{r}}
    e((r^j-r^g)t\eta_v).
\]
Our main tool is Lemma~\ref{3.15}, but it does not apply to all the terms  $A(x)$ in the sum.
So we will split it  into two smaller sums over $B(x)$ and $C(x)$, 
so that  a straightforward analysis applies to the first, and Lemma~\ref{3.15} applies to the other.
Let $p$ be the least integer satisfying the conditions for each $t\in T$, $r^{p-1} \geq 2|t|$
  and for each $r\in R$, $r^p\geq s^2+1$.
{\everymath={\displaystyle}
\begin{align*}
  B(x)=\sum_{t\in T}\sum_{r\in R}
  \left(
    \begin{array}{ll}
\sum_{g=\base{b}{r}-p+1}^{\base{b}{r}}
\sum_{j=\base{a}{r}+1}^{\base{b}{r}}
e((r^j-r^g)t x)&+\\
\sum_{g=\base{a}{r}+1}^{\base{b}{r}}
\sum_{j=\base{b}{r}-p+1}^{\base{b}{r}}
e((r^j-r^g)t x)&+\\
\sum_{g=\base{a}{r}+1}^{\base{b}{r}}
\sum_{\substack{j=\base{a}{r}+1\\ |g-j|<p}}^{\base{b}{r}}
e((r^j-r^g)t x).
    \end{array}
\right)
\end{align*}
}
Assume  for each $r\in R$, $\ell\geq \log r$ and $\ell\geq (8p\log s)^2$
 (and recall, $b=a+\ell$.) 
We obtain the following bounds. 
The first inequality uses that each term in the explicit definition of $B(x)$ has norm less than or equal to $1$.
The second uses the assumed conditions on $\ell$ and  the last inequality uses that $c(R,s)< 1/2$ as ensured by  Lemma~\ref{3.15}.
\begin{align}
  |B(x)|&\leq \sum_{t\in T} \sum_{r\in R}4p(\base{b}{r}-\base{a}{r}) 
  = {\card}T\, {\card}R \;4\, p\, (2\log s/\log r) (\base{b}{s}-\base{a}{s})\nonumber\\
  &\leq {\card}T\,{\card}R\; (\base{b}{s}-\base{a}{s})^{3/2} \nonumber\\
  &\leq {\card}T\,{\card}R\; (\base{b}{s}-\base{a}{s})^{2-c(R,s)/2}.\nonumber
\end{align}
Thus,
$\displaystyle{\sum_{v\in \L}B(\eta_v)\leq {\card}T\,{\card}R\; (b-a)^{2-c(R,s)/2}\,\stilde^{\base{b}{s}-\base{a}{s}}}$.
We estimate $\sum_{v\in \L}C(\eta_v)$, where
\[
C(x) =\sum_{t\in T}\sum_{r\in R}
\sum_{g=\base{a}{r}+1}^{\base{b}{r}-p}\;\;
\sum_{\substack{j=\base{a}{r}+1\\
    |j-g|\geq p}}^{\base{b}{r}-p}
e((r^j-r^g)t x).
\]
We will rewrite $C(x)$ conveniently. We start by rewriting $\sum_{v\in \L}A(\eta_v)$.
  \begin{align*}
    \sum_{v\in \L}A(\eta_v) 
   =&\sum_{t\in T}\;\sum_{r\in R}\;\sum_{v\in \L}\;
    \sum_{g=\base{a}{r}+1}^{\base{b}{r}}
    \sum_{j=\base{a}{r}+1}^{\base{b}{r}}
    e((r^j-r^g)t\eta_v)
  \\  =&
    \sum_{t\in T}\;\sum_{r\in R}\;
    \sum_{j=\base{a}{r}+1}^{\base{b}{r}}
    \sum_{g=\base{a}{r}+1}^{\base{b}{r}}\;
    \sum_{v\in \L}\;
    e((r^j-r^g)t\eta_v).
  \end{align*}
For fixed $j$ and $g$, we have the following identity.
\[
\sum_{v\in \L}\;\;
e((r^j-r^g)t\eta_v)=
\prod_{k=\base{a}{r}+1}^{\base{b}{r}}\Bigl(1+e\Bigl(\frac{t(r^j-r^g)}{s^k}\Bigl)+\dots+e\Bigl(\frac{(\stilde-1) t(r^j-r^g)}{s^k}\Bigl) \Bigl).
\]
Since $v\in\L= \digits{\stilde}{\base{b}{s}-\base{a}{s}}$
the digits in $v$ are in $\{0,\dots,\stilde-1\}$. Thus,
\[
\Bigl|\sum_{v\in \L} A(\eta_v)\Bigr|\leq
\sum_{t\in T}\sum_{r\in R}
\sum_{j=\base{a}{r}+1}^{\base{b}{r}}\;\;
\sum_{g=\base{a}{r}+1}^{\base{b}{r}}\;\;
\prod_{k=\base{a}{r}+1}^{\base{b}{r}}\;\Bigl|\;
\sum_{d=0}^{\stilde-1} e\Bigl(\frac{d t(r^j-r^g)}{s^k}\Bigr)
\;\Bigr|
\]
and
\[
\Bigl|\sum_{v\in \L} C(\eta_v)\Bigr|\leq\sum_{t\in T}\sum_{r\in R}
\sum_{j=\base{a}{r}+1}^{\base{b}{r}-p}
\sum_{\substack{g=\base{a}{r}+1\\
    |j-g|\geq p}}^{\base{b}{r}-p }\;\;
\prod_{k=\base{a}{r}+1}^{\base{b}{r}}\;
\Bigl|\;
\sum_{d=0}^{\stilde-1} e\Bigl(\frac{d t(r^j-r^g)}{s^k}\Bigr)
\;\Bigr|.
\]
\noindent Since $|\sum_x e(x)|=|\sum_x e(-x)|$, we can bound the sums over $g$ and $j$ as
follows.
\begin{equation}
  \nonumber
\Bigl|\sum_{v\in \L} C(\eta_v)\Bigr|\leq
2
\sum_{t\in T}\sum_{r\in R}
\sum_{j=p}^{\base{b}{r}-\base{a}{r}-p}\;\;
\sum_{g=1}^{\base{b}{r}-\base{a}{r}-p-j}\;\;
\prod_{k=\base{a}{r}+1}^{\base{b}{r}}\;
\Bigl|\;
\sum_{d=0}^{\stilde-1} e\Bigl(\frac{d tr^{\base{a}{r}}r^g(r^{j}-1)}{s^k}\Bigr)
\;\Bigr|.
\end{equation}
Let $L=(r^j-1) r^{\base{a}{r}}t$.  The following bounds related to $L$ are ensured by the
choice of $p$.  Let $\Tmax$ be the maximum of the absolute values of the elements of $T$.
\begin{align*}
  Lr^gs^{-\base{b}{s}}&\leq   (r^j -1)r^{\base{a}{r}} t r^g s^{-\base{b}{s}}\\
  &\leq r^j r^{\base{a}{r}} t r^{\base{b}{r} - \base{a}{r} - p - j}  s^{-\base{b}{s}} =
  t r^{\base{b}{r}-p} s^{-\base{b}{s}}\\
  &\leq \Tmax \; r^{\ceil{b/\log r}} \; s^{-\ceil{b/\log s}} r^{-p}\\
  &\leq \Tmax \; r^{1-p}\\
  &\leq 1/2 \qquad\mbox{(an ensured condition on $p$)}.
\end{align*}
We give a lower bound on the absolute value of $L$.
\begin{align*}
|L|&\geq (r^p-1)r^{\base{a}{r}} = (r^p-1)r^{\ceil{a/\log r}}\\
  &\geq 
(r^p-1) s^{a/\log s}\\
  &\geq s^{2+a/\log s} \qquad\mbox{(an ensured condition on $p$)}\\
\phantom{ Lr^gs^{-\base{b}{s}}} 
  &\geq s^{\base{a}{s} +1}.
\end{align*}
Below, we use 
$\Bigl|\sum_{d=0}^{\stilde-1} e({dx}) \Big|\leq (\stilde/2) \ |1+e(x)|$; notice that the leading coefficient is whole
(note to the curious reader: this  the only reason that $\stilde$ is required to be even).
\begin{align*}
\sum_{g=1}^{\base{b}{r}-\base{a}{r}-p-j}
\prod_{k=\base{a}{r}+1}^{\base{b}{r}}\;
\Bigl|\;
\sum_{d=0}^{\stilde-1} e(d L r^gs^{-k})
\;\Bigr|
&\leq
\sum_{g=1}^{\base{b}{r}-\base{a}{r}-p-j}
\prod_{k=\base{a}{r}+1}^{\base{b}{r}}\;
\frac{\stilde}{2}\, \Bigl|1+e\Bigl(r^g L s^{-k}\Bigr) \Bigr|
\end{align*}
which, by the double angle identities, is at most
$\displaystyle{
\stilde^{\base{b}{s} - \base{a}{s}}
\sum_{g=1}^{\base{b}{r}-\base{a}{r}-p-j}
\prod_{k=\base{a}{r}+1}^{\base{b}{r}}\;
|\cos(\pi L r^g s^{-k})|.}$
\newline
If  $k\geq \base{b}{r}$, then $L r^g s^{-k}\leq 2^{-(k+1)}$.  Therefore,
$\displaystyle{
\prod_{k=\base{b}{r}+1}^\infty |\cos(\pi L r^g s^{-k})|\geq \prod_{k=1}^\infty |\cos(\pi 2^{-(k+1)})|}$,
where the right hand side is a positive constant. Then,
\begin{align*}
\prod_{k=\base{a}{r}+1}^{\base{b}{r}} |\cos(\pi L r^g s^{-k})|
&=
\prod_{k=\base{a}{r}+1}^\infty |\cos(\pi L r^g s^{-k})|
\;\;\left(\prod_{k=\base{b}{r}+1}^\infty |\cos(\pi L r^g s^{-k})|\right)^{-1}
\end{align*}
which, for the appropriate constant $\constant$, is at most
$\displaystyle{   \constant \;\prod_{k=\base{a}{r}+1}^\infty |\cos(\pi L r^g s^{-k})|}$.
\newline
We can apply Lemma~\ref{3.15}:
\begin{align*}
\sum_{g=1}^{\base{b}{r}-\base{a}{r}-p-j}
\prod_{k=\base{a}{r}+1}^{\base{b}{r}}\;
\Bigl|\;
\sum_{d=0}^{\stilde-1} e(d L r^gs^{-k})
\;\Bigr|
&\leq
\sum_{g=1}^{\base{b}{r}-\base{a}{r}-p-j}
\constant \;\prod_{k=\base{a}{r}+1}^\infty |\cos(\pi L r^g s^{-k})|\\
&\leq 2\constant({\base{b}{r}-\base{a}{r}})^{1-c(R,s)}.
\end{align*}
\begin{align*}
\Bigl|\sum_{v\in \L} C(\eta_v)\Bigr|&\leq
2
\sum_{t\in T}\sum_{r\in R}
\sum_{j=p}^{\base{b}{r}-\base{a}{r}-p}\;\;
\stilde^{\base{b}{s}-\base{a}{s}}\; 2\constant ({\base{b}{s}-\base{a}{s}})^{1-c(R,s)}\\
&\leq
4\constant\; {\card} T\;{\card} R\;({\base{b}{s}-\base{a}{s}})^{2-c(R,s)}\; \stilde^{\base{b}{s}-\base{a}{s}}.
\end{align*}
Combining this with the estimate for $|\sum_{v\in\L} B(\eta_v)|$, we have
\[
|\sum_{v\in\L} A(\eta_v)|\leq 4\constant\;  {\card}T {\card}R
(\base{b}{s}-\base{a}{s})^{2-c(R,s)}\stilde^{\base{b}{s}-\base{a}{s}}. 
\]
Therefore, the number of $v\in\L$ such that
$\displaystyle{A(\eta_v)> 4\constant\;  {\card}T\; {\card}R (\base{b}{s}-\base{a}{s})^{-c(R,s)/2}}$
is at most
$(\base{b}{s}-\base{a}{s})^{-c(R,s)/2}\; \stilde^{\base{b}{s}-\base{a}{s}}.$
If $\ell>(2^{2/c(R,s)}+1)\log s$ and $\ell>(16\constant\card{T}\card{R})^{4/c(R,s)}$ then
$(\base{b}{s}-\base{a}{s})^{-c(R,s)/2}<1/2$. So, there are at least 
$(1/2)\stilde^{(\base{b}{s}-\base{a}{s})}$ members $v\in\L$ for which
\begin{align*}
  A(\eta_v)&\leq 4\constant\;  {\card}T\, {\card}R (\base{b}{s}-\base{a}{s})^{2-c(R,s)/2}\\
  &\leq 4\constant\;  {\card}T\, {\card}R (2\ell)^{2-c(R,s)/2}\\
&\leq \ell\,^{2-c(R,s)/4}.
\end{align*}
This proves the lemma for $\ell_0$ equal to the least integer greater than  
$(2^{2/c(R,s)}+1)\log s$,
$(16\constant\card{T}\card{R})^{4/c(R,s)}$, 
$(8p\log s)^2$ and 
$\max\{\log r : r\in R\}$.  
\end{proof}

\subsection{On Changing Bases}

\begin{lemma}\label{3.18}
  For any interval $I$ and base $s$, there is a $s$-adic subinterval $I_s$ such that
  $\measure{I_s}\geq \measure{I}/(2s)$ .
\end{lemma}

\begin{proof}
  Let $m$ be least such that $1/s^m<\measure{I}$.
  Note that $1/{s^m}\geq \measure{I}/ {s}$, since $1/s^{m-1}\geq\measure{I}$.  
  If there is a $s$-adic interval of length $1/ {s^m}$ strictly contained in $I$, 
  then let $I_s$ be such an interval, and note that $I_s$ has
  length greater than or equal to $\measure{I}/{s}$.  
  Otherwise, there must be an $a$ such that
  $a/s^m$ is in $I$ and neither $(a-1)/s^m$ nor $(a+1)/s^m$ belongs to $I$.  
  Thus, $2/s^m>\measure{I}$.  
  However, since $1/s^m < \measure{I}$ and $s\geq 2$ then 
  $2/s^{m+1}<\measure{I}$.  
   So, at least one of the two intervals
  $\displaystyle{\left[\frac{sa-1}{s^{m+1}}, \frac{sa}{s^{m+1} }\right)}$ or
  $\displaystyle{\left[\frac{sa}{s^{m+1} }, \frac{sa+1}{s^{m+1} }\right)}$ must be contained in $I$.
  Let $I_s$ be such. 
  Then, $\measure{I_s}$ is $\displaystyle{\frac{1}{s^{m+1}}=\frac{1}{2s}\frac{2}{s^m}\ >\  \measure{I}/(2s).}$ 
  In either case, the length of $I_s$ is greater than $\measure{I}/(2s)$.
\end{proof}

\begin{lemma}\label{3.19}
  Let $s_0$ and $s_1$ be bases and suppose that $I$ is an $s_0$-adic interval of length
 $s_0^{-\<b;s_0\>}$.  For $a=b+\ceil{\log s_0 + 3\log s_1}$, there is an $s_1$-adic subinterval of
 $I$ of length $s_1^{-\<a;s_1\>}$.  
 \end{lemma}

\begin{proof}
By the proof of Lemma~\ref{3.18}, there is an $s_1$-adic subinterval of $I$ of length
$s_1^{-(\ceil{-\log_{s_1}(\mu(I))}+1)}$:
\begin{align*}
  \ceil{-\log_{s_1}(\measure{I})}+1&= \ceil{-\log_{s_1}(s_0^{-\base{b}{s_0}})}+1
=\ceil{
    {\base{b}{s_0}\log s_0 }/{\log s_1}
      }+1\\
  &\leq \ceil{ b/\log s_1+ \log s_0/\log s_1}+1\\
  &\leq \base{b}{s_1} +\ceil{\log s_0/\log s_1} +1.  
\end{align*}

Thus, there is an $s_1$-adic
subinterval of $I$ of length $s_1^{-(\base{b}{s_1} +\ceil{\log s_0/\log s_1} +1)}$.  
Consider
$a=b+\ceil{\log s_0 + 3\log s_1}$.  Then
\begin{align*}
  \base{a}{s_1} &= \ceil{a/\log s_1} = \ceil{{b+\ceil{\log s_0 + 3\log s_1}}/{\log s_1}}\\
  &\geq b/\log s_1 +(\log s_0+3\log s_1)/\log s_1\\
  &\geq \base{b}{s_1} +\ceil{\log s_0/\log s_1} +1.
\end{align*}
This inequality is sufficient to prove the lemma.
\end{proof}

The next observation is by direct substitution. We will use it in the proofs of the theorems.

\begin{remark}\label{3.20}
   Suppose that $r$, $s_0$ and $s_1$ are bases.  Let $b$ be a positive integer and let
   $a=b+\ceil{\log s_0 + 3\log s_1}$.  Then,
   $\base{a}{r}-\base{b}{r}\leq \ceil{\log s_0 + 3\log s_1}/\log r+1.$  Hence,\linebreak
   $\base{a}{r}-\base{b}{r}\leq 2\ceil{\log s_0 + 3\log s_1}.$  
 \end{remark}

\section{Proofs of Theorems}

\subsection{Tools}

\begin{notation}
Let $M$ be the set of minimal representatives of the multiplicative dependence equivalence 
classes.  Let $p(s_0,s_1)=2\ceil{\log s_0+3\log s_1}$.  
\end{notation}

\begin{definition}\label{4.1}
  Let $T$ and $\delta$ be as defined  in Lemma~\ref{3.8} for input $(\epsilon/10)^4$.
Let $\ell$ be the function with inputs $R$, $s$, $k$, $\epsilon$ and value the least integer greater than all  of the following:  
\begin{itemize}
\item The maximum of $\ell_0$ as defined in Lemma~\ref{3.9} over all  inputs $r$ in $R$ 
  and $\epsilon$ as given.  
\item  $N_0$ as defined  in Lemma~\ref{3.14} for  inputs $s$, $k$ and $\epsilon$ 
\item  $\ell_0$ as defined in  Lemma~\ref{3.17} for inputs $R$, $T$ and $s^k$.
\item $\left((\log r)^2/\delta\right)^{4/c(R,s^k)}$ for $c(R,s^k)$ the minimum of the constants of Lemma~\ref{3.15} for pairs $s,r$ with $r\in R$.
\end{itemize}
\end{definition}

\subsection{Proof of Theorem \ref{2}}

\begin{theoremD}
  For any $\Pi^0_3$ subset $R$ of $M$ there is a computable real number $\xi$ such that for all
  $r\in M$, $r\in R$ if and only if $\xi$ is normal to base $r$.  Furthermore, $\xi$ is computable
  uniformly in the $\Pi^0_3$ formula that defines $R$.
\end{theoremD}

Note that $m\in M$ if and only if there is no $n$ less than $m$ such that $m$ is an integer power of~$n$, 
an arithmetic condition expressed using only bounded quantification.
Let $\phi=\forall x\exists y\forall z\theta$ be a $\Pi^0_3$ formula with one free variable.  
We will construct a real number $\xi$ so that for every base $r$, $\xi$ is normal to base $r$ if and only if
$\phi(r)$ is true.  The normality of $\xi$ to base $r$ is naturally expressed using three
quantifiers: $\forall\epsilon\exists n\forall N\geq n\  D(\expa{r^k\xi}: 0\leq k<N)<\epsilon$.
Lemma \ref{3.1} shows that the discrepancy $D$
 admits computable  approximations (using finite partitions of the unit interval).
Thus, the normality of $\xi$ to base $r$  is a $\Pi^0_3$ formula. 
In our construction, we will bind the quantified variables in $\phi(r)$
to those in the formula for normality.  
The variable $x$ will correspond to $\epsilon$, $y$ to $n$ and $z$ to $N$.

We define a sequence $\xi_m$, $b_m$, $s_m$, $k_m$, $\epsilon_m$, $\ell_m$, $x_m$, $R_m$ and $c_m$ by
stages.  $\xi_m$ is a $s_m^{k_m}$-adic rational number of precision $\base{b_m}{s_m^{k_m}}$.  $b_m$
and $k_m$ are positive integers.  $s_m$ is a base.  $R_m$ is a finite set of bases.  The real $\xi$
will be an element of $[\xi_m,\xi_m+(s_m^{k_m})^{-\base{b_m}{s_m^{k_m}}})$.  Stage $m+1$ is devoted
to extending $\xi_{m}$ so that the discrepancy of the extended part in base $s_{m+1}$ is below $1/x_m$
and above $1/(2s_{m+1}^{k_{m+1}})$, and so that the discrepancy of the extension 
for the other bases under consideration is below $\epsilon_{m+1}$. 
$\ell_{m+1}$ is used to determine the
length of the extension and $c_{m+1}$ is an integer used to monitor $\phi$ and set bounds on
discrepancy.
Fix an enumeration of $M$ such that every element of $M$ appears infinitely often.
\medskip

\noindent\emph{Initial stage.}  Let $\xi_0=0$, $b_0=1$, $s_0=3$, $k_0=1$, $\epsilon_0=1$,
$\ell_0=1$, $x_0=1$ and $c_0=1$.

\medskip\noindent{\em Stage $m+1$.\ } 
Given 
$\xi_{m}$ of the form $\sum_{j=1}^{\base{b_{m}}{s_m^{k_{m}}}} v_j(s_m^{k_{m}})^{-j}$,
$b_{m}$, 
$s_{m}$,
$k_{m}$,
$\epsilon_{m}$,
$\ell_m$,
$x_m$,
$R_m$
and 
$c_m$.

(1) Let $F$ be the canonical partition of $[0,1]$ into intervals of length $(1/3)(1/(4)s_m^{-k_m}$.
If~$D(F,(\expa{s_m^j\xi_m}:0\leq j< \base{b_m}{s_m}))<((1/3)(1/4)s_m^{-k_m}))^2$, then let 
$s_{m+1}$ be $s_m$, 
$k_{m+1}$ be $k_m$, 
$\epsilon_{m+1}$ be $\epsilon_m$, 
$\ell_{m+1}$ be $\ell_m$, 
$x_{m+1}$ be $x_m$, 
$R_{m+1}$ be $R_m$ and 
$c_{m+1}$ be $c_m$.
\medskip

(2) Otherwise, let $c$ be $c_m+1.$ Let $s$ be the $c$th element in the enumeration of $M$. Let $n$ be
maximal less than $c$ such that $s$ is also the $n$th element in the enumeration of $M$, or be 0 if
$s$ appears for the first time at $c$.  Take $x$ to be minimal such that there is a $y$ less than
$n$ satisfying $\forall z<n\,\phi(x,y,z)$ and $\exists z<c\,\neg\phi(x,y,z)$.  If there is none
such, then set $x$ equal to $c$.  
Let $k$ and $N$ be as defined in Lemma~\ref{3.14} for input $\epsilon=1/x$ and base $s$.  
Let $R$ be the set of bases not equal to $s$ which appear in the enumeration of $M$ at positions less than $c$.  
Let $L$ be the least integer greater than 
$\max\{x, c ,  2 s^{k}\} \log(\max(R\cup\{s\})) p(s_m,s)$,
$N$, and 
$\ell(R,s,k,1/c)$. 
If for some $r\in R$, 
$(1/c) \base{b_m}{r}\leq L+p(s_m,s)$ or $(1/x) \base{b_m}{s}\leq L+p(s_m,s)$ then 
let $s_{m+1}$ be $s_m$, 
 $k_{m+1}$ be $k_m$, 
 $\epsilon_{m+1}$ be $\epsilon_m$, 
 $\ell_{m+1}$ be $\ell_m$, 
 $x_{m+1}$ be $x_m$, 
 $R_{m+1}$ be $R_m$ and 
 $c_{m+1}$ be $c_m$.
\medskip

(3) Otherwise, 
let $s_{m+1}$ be $s$, 
$k_{m+1}$ be $k$, 
$\epsilon_{m+1}$ be $1/c$, 
$\ell_{m+1}$ be $L$  
$x_m$ be $x$, 
$R_{m+1}$ be $R$ and 
$c_{m+1}$ be $c$.
\smallskip
 
Let $a_{m+1}$ be minimal such that there is an $s_{m+1}^{k_{m+1}}$-adic subinterval of
$[\xi_{m},\xi_{m}+(s_{m}^{k_m})^{-\base{b_{m}}{s_{m}^{k_m}}})$ of length $(s_{m+1}^{k_{m+1}})^{-\base{a_{m+1}}{s_{m+1}^{k_{m+1}}}}$
and let  
$[\eta_{m+1},\eta_{m+1}+(s_{m+1}^{k_{m+1}})^{-\base{a_{m+1}}{s_{m+1}^{k_{m+1}}}})$ be the leftmost such.
Let $\stilde$ be $s_{m+1}^{k_{m+1}}-1$ if $s_{m+1}$ is odd and be $s_{m+1}^{k_{m+1}}-2$ otherwise.  
Let $T$ and $\delta$ be as defined in Lemma~\ref{3.8} for input $\epsilon=(\epsilon_{m+1}/10)^4$.
Let~$b_{m+1}$  be $a_{m+1}+\ell_{m+1}$.
Let $\nu$ be such that
\begin{itemize}
\item $\displaystyle{ 
    \nu=\eta_{m+1}+\sum_{j=\base{a_{m+1}}{s_{m+1}^{k_{m+1}}}+1}^{\base{b_{m+1}}{s_{m+1}^{k_{m+1}}}}  w_j(s_{m+1}^{k_{m+1}})^{-j}}$,
  for some $(w_1,\dots,w_{\ell_{m+1}})$ in $\digits{\stilde}{\ell_{m+1}}$.

\item $A(\nu,R_{m+1},T,a_{m+1},\ell_{m+1})/\base{\ell_{m+1}}{\max(R_{m+1})}^2<\delta$.

\item $\nu$ minimizes
  $D(F,(\expa{s_{m+1}^j\nu}:\base{a_{m+1}}{s_{m+1}}<
  j\leq \base{b_{m+1}}{s_{m+1}}))$ among the $\nu$ satisfying the first two conditions, for $F$ as
  defined in clause  (1).  If there is more than one minimizer, take the least such for $\nu$.
\end{itemize}
\nopagebreak[4] We define $\xi_{{m+1}}$ to be $\nu$.  This ends the
description of stage $m+1$.  
\medskip

We verify that the construction succeeds.  
Let $m+1$ be a stage.  If clause (1) or (2) applies, let $m_0$ be the greatest stage less than or
equal to $m+1$ such that $c_{m_0}=c_{m_0+1}=\cdots=c_{m+1}$.  During stage $m_0$, $k_{m_0}$ and
$\ell_{m_0}$ were chosen to satisfy the conditions to reach clause (3).  
Note that since
$b_m>b_{m_0}$ the conditions in clause (2) apply to $b_m$ in place of $b_{m_0}$: 

$(1/c_{m+1}) b_m > \ell_{m+1} +p(s_{m_0-1},s_{m+1})$ and
$(1/x_{m+1}) b_m > \ell_{m+1} +p(s_{m_0-1},s_{m+1})$. 
Then,  $\ell_{m+1}$ is greater than 
$\max\{x_{m+1},  c_{m+1}, 2 s_{m+1}^{k_{m+1}}\} \log(\max(R_{m+1}\cup\{s_{m+1}\}) p(s_{m_0-1},s_{m+1})$,
$N$, and 
$\ell(R_{m+1},s_{m+1},k_{m+1},1/c_{m+1})$, where $N$ is determined during stage $m_0$.  
If clause (3) applies, then the analogous conditions hold by construction.

Stage $m+1$ determines the $s_{m+1}^{k_{m+1}}$-adic subinterval
$[\eta_{m+1},\eta_{m+1}+(s_{m+1}^{k_{m+1}})^{-\base{a_{m+1}}{s_{m+1}^{k_{m+1}}}})$ of the interval
provided at the end of stage $m$.  The existence of this subinterval is ensured by Lemma~\ref{3.19}.
The stage ends by selecting the rational number $\nu$.  The existence of an appropriate $\nu$ is
ensured by Lemma~\ref{3.17} with the inputs given by the construction.  It follows that $\xi$ is
well defined as the limit of the $\xi_{m}$.

Let $s$ be a base that appears in the enumeration of $M$ at or before $c_{m+1}$.  
There are two possibilities for $s$ during stage $m$: either it is an element of $R_{m+1}$ or it is
equal to $s_{m+1}$.
Suppose first that $s\in R_{m+1}$.  $\xi_{m+1} = \nu$ was chosen so that
$A(\nu,R_{m+1},T,a_{m+1},\ell_{m+1})/\base{\ell_{m+1}}{s}^2<\delta$.  
By~Definition~\ref{3.16},
$\displaystyle{ A(\nu,R_{m+1},T,a_{m+1},\ell_{m+1}) }$ is equal to $\displaystyle{\sum_{t\in  T}\;\sum_{r\in R_{m+1}}\; \Bigl|\sum_{j=\base{a_{m+1}+1}{r}}^{\base{b_{m+1}}{r}} 
 e(r^j t \nu)\Bigr|^2 }$.
Hence,
$\displaystyle{(1/\base{\ell_{m+1}}{s}^2) \sum_{t\in T}\,\Bigl|\sum_{j=\base{a_{m+1}+1}{s}}^{\base{b_{m+1}}{s}} e(s^j t \nu)\Bigr|^2<\delta.
}$ 
By choice of $T$ and $\delta$, Lemma~\ref{3.8} ensures
\[
D(s^j\nu:\base{a_{m+1}}{s}<j\leq \base{b_{m+1}}{s})<(\epsilon_{m+1}/10)^4.
\]
By definition of $\xi$, $\xi \in[\nu,\nu+(s_{m+1}^{k_{m+1}})^{-\base{b_{m+1}}{s_{m+1}^{k_{m+1}}}})$.  By Lemma~\ref{3.9}, we
conclude that
\[
D(s^j\xi:\base{a_{m+1}}{s}< j\leq \base{b_{m+1}}{s})<\epsilon_{m+1}.
\]
By Remark~\ref{3.20}, $\base{a_{m+1}}{s}-\base{b_m}{s}$ is less than or equal
to $p(s_m,s_{m+1})$.  By construction, $\epsilon_{m+1}\,\base{\ell_{m+1}}{s}$ is greater than
$p(s_m,s_{m+1})$.  By Lemma~\ref{3.3}
\[
D(s^j\xi:\base{b_{m}}{s}< j\leq \base{b_{m+1}}{s})<2\epsilon_{m+1}.
\]

The second possibility is that $s$ is equal to $s_{m+1}$.  Again, consider the selection of the rational number $\nu$
during stage $m+1$.  By Lemma~\ref{3.17}, more than half of the eligible candidates satisfy the
inequality
$A(\nu,R_{m+1},T,b_{m},\ell_{m+1})/\base{\ell_{m+1}}{\max(R_{m+1})}^2<\delta$.  
By Lemma~\ref{3.14}, more than half the candidates satisfy
\[
D(\expa{s^j\nu}:\base{a_{m+1}}{s}< j\leq \base{b_{m+1}}{s})<1/x_{m+1}.
\]
By choice of $\xi_{m+1}$,
\[
D(F,(\expa{s^j\xi_{m+1}}:\base{a_{m+1}}{s}< j\leq
\base{b_{m+1}}{s})<1/x_{m+1}.
\] 
By Lemma~\ref{3.1},
\[
D(\expa{s^j\xi_{m+1}}:\base{a_{m+1}}{s}< j\leq
\base{b_{m+1}}{s})<3 x_{m+1}^{-1/2}.
\]
By construction $(1/x_{m+1}) \base{\ell_{m+1}}{s}$ is greater than $p(s_m,s)$ and so, as
above,
\[
D(\expa{s^j\xi_m}:\base{b_{m}}{s}< j\leq
\base{b_{m+1}}{s})<6 x_{m+1}^{-1/2}.
\]
By Lemma~\ref{3.9}, 
\[
D(\expa{s^j\xi}:\base{b_{m}}{s}< j\leq
\base{b_{m+1}}{s})<10(6x_{m+1}^{-1/2})^{1/4}.
\]
Further, $\xi_{{m+1}}$ is chosen so that in 
$(\expa{(s^{k_{m+1}})^j\xi_{m+1}}:\base{a_{m+1}}{s^{k_{m+1}}}<j\leq\base{b_{m+1}}{s^{k_{m+1}}})$  no
element belongs to
$[1-s^{k_{m+1}},1]$.  As $\xi\in[\xi_{m+1},\xi_{m+1}+s^{-\base{b_{m+1}}{s}}),$ the same holds for $\xi$. 
Since $\base{b_{m+1}}{s^{k_{m+1}}}-\base{b_{m}}{s^{k_{m+1}}}>2 s^{k_{m+1}} p(s_m,s)$, 
Lemma~\ref{3.5} applies and so
\[
D(\expa{s^j\xi}:\base{b_{m}}{s}<j\leq\base{b_{m+1}}{s}) )\geq 1/(2s^{k_{m+1}}).
\]
Stages subsequent to $m+1$ will satisfy the same inequality until the first stage for which
$D(\expa{s_m^j\xi_m}:0\leq j< \base{b_m}{s_m})\geq 1/(4s_m^{k_m})$.  By a
direct counting argument, there will be such a stage and during that stage clause~(1) cannot apply.
Similarly, clause (2) cannot apply for indefinitely many stages, as the values of $b_m$ are
unbounded.  It follows that $\lim_{m\to\infty}c_{m+1}=\infty$.

If $\phi(s)$ is true, then for each $x$, there are only finitely many stages during which $s_m=s$
and $x_m=x$.  Let $\epsilon$ be greater than $0$.  There will be a stage $m_0$ such that for all $m$
greater than $m$, $\epsilon>2\epsilon_m$ and, if $s=s_m$ then $\epsilon>10(6x_{m+1}^{-1/2})^{1/4}$.
By construction, Lemma~\ref{3.4} applies to conclude $\lim_{N\to\infty}D( \expa{s^j\xi_m}:0\leq
j< N)\leq 2\epsilon$.  By applying Lemma~\ref{3.9}, we conclude that $\xi$ is normal to base $s$.

If $\phi(s)$ is not true, then let $x$ be minimal such that $\forall y \exists z\neg\phi(s,x,y,z)$.
There will be infinitely many $m+1$ such that $s=s_{m+1},$ $x=x_{m+1}$ and $k_{m+1}$ is the $k$
associated with $s$ and~$x$.  As already discussed, each of these stages will be followed by a later
stage $m_1$ such that
\[
D(\expa{s^j\xi}:0\leq j< \base{b_{m_1}}{s})\geq 1/(4s^{k}).
\]
Hence $\xi$ is not normal to base $s$.

\subsection{Proof of Theorem \ref{1}}

\begin{theoremU}
(1) The set  of indices for computable real numbers which are normal to at least one base is $\Sigma^0_4$-complete. 
(2) The set of real numbers that are normal to at least one base is $\mathbf \Sigma^0_4$-complete.
\end{theoremU}

To prove item (1) we must exhibit a computable function $f$, taking $\Sigma^0_4$ sentences (no free
variables) to indices for computable real numbers, such that for any $\Sigma^0_4$ sentence $\psi$, $\psi$
is true in the natural numbers if and only if the computable real number named by $f(\psi)$ is normal to
at least one base.  Let $\psi$ be a $\Sigma^0_4$ sentence and let $\phi$ be the $\Pi^0_3$ formula such
that $\psi=\exists w\phi(w)$.
Let $M$ be the set of minimal representatives of the multiplicative dependence equivalence
classes and fix the computable enumeration of $M=\{s_1,s_2,\dots\}$ (as in the proof of Theorem~\ref{2}).  Consider
the $\Pi^0_3$ formula $\phi^*$ such that $\phi^*(s_w)$ is equivalent to $\phi(w)$.  By
Theorem~\ref{2}, there is a computable real $\xi$ such that for all $s_w$, $\xi$ is normal to base
$s_w$ if and only if $\phi^*(s_w)$ is true, if and only if $\phi(w)$ is true.  Thus, $\xi$ is normal
to at least one base if and only if there is a $w$ such that $\phi(w)$ is true, if and only if 
$\psi=\exists w\phi(w)$ is true.  In Theorem~\ref{2}, $\xi$ is obtained uniformly from $\phi^*$,
which was obtained uniformly from $\phi$.  The result follows.

For item (2)  recall that a subset in $\R$  is $\bSigma^0_4$-complete if
it is $\bSigma^0_4$ and it is hard for $\bSigma^0_4$.  To prove hardness of subsets of $\R$ at levels in
the Borel hierarchy it is sufficient to consider subsets of Baire space, $\N^\N$ because there is a
continuous function from $\R$ to $\N^\N$ that preserves the levels. 
The Baire space admits a syntactic representation of the levels of Borel hierarchy in arithmetical terms. 
A subset $A$ of $\N^\N$ is $\bSigma^0_4$ if and only if 
there is a  parameter $p$ in $\N^\N$ and 
a $\Sigma^{p}_4$ formula $\psi(x,p)$, where $x$ is a free variable,
such that for all $x\in\N^\N$, $x\in A$ if and only if $\psi(x,p)$ is true.  
A subset $B$ of $\R$ is hard for $\bSigma^0_4$ if for every 
$\bSigma^0_4$ subset $A$ of $\N^\N$ there is a
continuous function $f$ such that for all 
$x\in\N^\N$, $x\in A$ if and only if $f(x)\in B$.
Consider a $\bSigma^0_4$ subset $A$ of the Baire space
defined by a $\Sigma^{p}_4$ formula $\psi(x,p)$, where $x$ is a free variable. 
The same function given for item (1) but now relativized to $x$ and $p$  yields a real number  $\xi$
such  that  $\psi(x,p)$ is true if and only if   $\xi$ is normal to at least one base.
This gives  the required  continuous function $f$ satisfying $x\in A$ if and only $f(x)$ is normal to at least one base.

\subsection{Proof of Theorem \ref{3}}

\begin{theoremCi}
  For any $\Pi^0_3$ formula $\phi$ there is a computable real number $\xi$ such that for any base $r\in M$,
  $\phi(\xi,r)$ is true if and only if $\xi$ is normal to base $r$.
\end{theoremCi}

The proof  follows from Theorem~\ref{2} by an application of the Kleene Fixed Point Theorem~\cite[see][Chapter~11]{Rog87}.
Let $\phi$ be a $\Pi^0_3$ formula with two free variables, one ranging over $\N^\N$ and the other
ranging over $\N$.  Let $\Psi_e$ be a computable enumeration of the partial computable functions
from $\N$ to $\N$.  The condition ``$\Psi_e$ is a total function and $\phi(\Psi_e,r)$'' is a
$\Pi^0_3$ property of $e$ and $r$.  By Theorem~\ref{2}, there is a computable function which on
input a $\Pi^0_3$ formula $\theta$ produces a (total) computation of a real $\xi_\theta$ which is
normal to base $r\in M$ if and only if $\theta(r)$ is true.  In particular, there is a computable
function $f$ such that for every $e$, for all $r\in M$,
\[
\Psi_e \text{ is a total function and } \phi(\Psi_e,r) \text{ if and only if } \Psi_{f(e)} \text{is
  normal to base $r$.}
\]
Furthermore, for every $e$, $\Psi_{f(e)}$ is total.  By the Kleene Fixed Point Theorem, there is an
$e$ such that $\Psi_e$ is equal to $\Psi_{f(e)}.$ For this $e$, for all $r\in M$,
\[
\phi(\Psi_e,r) \text{ if and only if } \Psi_e \text{ is normal to base $r$.}
\]
Then, $\xi=\Psi_e$ satisfies the condition of the Theorem.

\subsection{Proof of Theorem \ref{4}}

\begin{theoremT}
  Fix a base $s$.  There is a computable function $f:\N\to\Q$ monotonically decreasing to $0$ such
  that for any function $g:\N\to\Q$ monotonically decreasing to $0$ there is an absolutely normal
  real number $\xi$ whose discrepancy for base $s$ eventually dominates $g$ and whose discrepancy
  for each base multiplicatively independent to $s$ is eventually dominated by $f$.  Furthermore,
  $\xi$ is computable from $g$.
\end{theoremT}

Let $s$ be a base. 
We define a sequence $\xi_m$, $b_m$, $k_m$, $\epsilon_m$, $\ell_m$, $R_m$ and $\kbar_m$ by stages.
$b_m$, $k_m$ and $\kbar_m$  are a positive integers, $\epsilon_m$ a positive rational number and $R_m$ a finite
set of bases multiplicatively independent to $s$.  $\xi_m$ is an $s^{k_m}$-adic rational number of
precision $\base{b_m}{s^{k_m}}$.  The real $\xi$ will be an element of
$[\xi_m,\xi_m+(s^{k_m})^{-\base{b_m}{s^{k_m}}})$.  Stage $m+1$ is devoted to extending $\xi_{m}$ so
that the discrepancy of the extension is below $\epsilon_{m+1}$ for the bases in $R_{m+1}$ and so
that the discrepancy of the extension in base $s$ is in a controlled interval above $g$.  We use
$k_{m+1}$ to enforce the endpoints of this interval. 
$\ell_m$ determines the length of the extension.

At each stage $m$ the determination of $\xi_{m+1}$ is done so that the discrepancy functions for
$\xi$ relative to bases independent to $s$ converge to $0$ uniformly, without reference to the
function~$g$.  We obtain $f$ as the function bounding these discrepancies by virtue of construction.
The variable $\kbar_m$ acts as a worse case surrogate for the exponent of $s$ used in the
construction relative to~$g$.  

\medskip

\noindent{\em Initial stage.\ } Let $\xi_0=0$, $b_0=1$, $k_0=1$, $\epsilon_0=1$, $\ell_0=0$, $R_0=\{r_0\}$ where $r_0$ is the
least base which is multiplicatively independent to $s$, and $\kbar_0=1$.

\medskip\noindent{\em Stage $m+1$.\ } 
Given 
$b_{m}$, 
$R_{m}$, 
$\epsilon_{m}$, 
$\kbar_{m}$,
$k_{m}$ and 
$\xi_{m}$ of the form $\sum_{j=1}^{\base{b_{m}}{s^{k_{m}}}}
  v_j(s^{k_{m}})^{-j}$.

  (1) Let $\nextr$ be the least number greater than the maximum element of $R_{m}$ which is
  multiplicatively independent to $s$.  If 
  $(\epsilon_m/2)\base{b_{m}}{\nextr}\geq\ell(R_{m}\cup\{\nextr\},s,\kbar_{m}+1,\epsilon_{m}/2)$ then let
  $\epsilon_{m+1}$ be $\epsilon_{m}/2$, let $R_{m+1}$ be $R_{m}\cup\{\nextr\}$ and $\kbar_{m+1}$ be
  $\kbar_{m}+1$.  Otherwise, let $\epsilon_{m+1}$ be $\epsilon_{m}$, $R_{m+1}$ be $R_{m}$ and
  $\kbar_{m+1}$ be $\kbar_{m}$.  Let $\ell_{m+1}=\ell(R_{m+1},s,\kbar_{m+1},\epsilon_{m+1})$ and let
  $b_{m+1}=b_{m}+\ell_{m+1}$.
\medskip

(2) Let $\betterk$ and $\betterN$ be as determined by Lemma~\ref{3.14} for the input value
$\epsilon=1/(4s^{k_{m}})$.  If $(\betterk\leq \kbar_{m+1})$, $(\betterN\leq\base{\ell_{m+1}}{s})$
and $(1/(2s^\betterk)>g(\base{b_{m}}{s}))$, then let $k_{m+1}$ be $\betterk$.  Otherwise, let
$k_{m+1}$ be $k_{m}$.  
\medskip

We define $\xi_{{m+1}}$ to be $\xi_{m}+\nu$, where $\nu$ is determined as follows.  
Let $\stilde$ be $s^{k_{m+1}}-1$ if $s$ is odd and be $s^{k_{m+1}}-2$ if $s$ is even. 
Let $T$ and $\delta$ be as determined in Lemma~\ref{3.8} with input $\epsilon=(\epsilon_{m+1}/10)^4$. 
Let $\nu$ be such that
    \begin{itemize}
    \item $\displaystyle{ 
    \nu=\xi_{m}+\sum_{j=1}^{\base{\ell_{m+1}}{s^{k_{m+1}}}}  w_j(s^{k_{m+1}})^{-(\base{b_{m}}{s^{k_{m+1}}}+j)}}$
    for some $(w_1,\dots,w_{\base{\ell_{m+1}}{s_{m+1}^{k_{m+1}}}})$ in\\ 
    $\digits{\stilde}{\base{\ell_{m+1}}{s_{m+1}^{k_{m+1}}}}$.

    \item $A(\nu,R_{m+1},T,\ell_{m+1})/\base{\ell_{m+1}}{\max(R_{m+1})}^2<\delta$.

  \item $\nu$ minimizes $D(F,(\expa{s^j\nu}:0\leq j<\base{\ell_{m+1}}{s}))$ among the $\nu$
    satisfying the first two conditions, where $F$ is the canonical partition of $[0,1]$ into
    intervals of length $(1/3)(1/4s^{k_{m+1}})$.  If there is more than one minimizer, take the
    least such for $\nu$.
  \end{itemize}

  We define the function $f:\N\to\Q$ as follows.  Given a positive integer $n$, let $m_n$ be such
  that $b_{m_n}\leq n<b_{m_n+1}$. Let $m_0$ be maximal such that
  $\epsilon_{m_0}\base{b_{m_n}}{\max(R_{m_0})}>b_{m_0}$.  Define $f(n)$ to be $4\epsilon_{m_0}$.  By
  construction, $\epsilon_m$ is monotonically decreasing and so $f$ is also.  Note, for all $m$,
  $\ell_m>0$ and $\lim_{m\to\infty}b_m=\infty$.  For every stage $m+1$, clause (1) sets
  $\epsilon_{m+1}$ to be $\epsilon_m/2$, unless $b_m$ is not sufficiently large.  The value of
  $\epsilon_{m+1}$ will be reduced at a later sufficiently large stage.  Thus, $\epsilon_m$ goes to
  $0$ and so does $f$.

The function $f$ is defined in terms of the sequences of values $b_m$, $R_m$ and $\epsilon_m$, which
are determined by clause (1).  The conditions and functions appearing in clause (1) are computable,
as was verified in each of the relevant lemmas.  Thus, $f$ is a computable function.

Suppose that $r$ and $s$ are multiplicatively independent.  Fix $n_0$ and $n_1$ so that $r\in
R_{n_0}$ and $\epsilon_{n_0}\base{b_{n_1}}{\max(R_{n_0})}>b_{n_0}$.  Let $n$ be any integer greater than
$b_{n_1}$ let $m_n$ be such that $b_{m_n}\leq n<b_{m_n+1}$.  By definition of $f$, there is an $m_0$ such
that $f(n)=4\epsilon_{m_0}$ and $\epsilon_{m_0}\base{b_{m_n}}{\max(R_{m_0})}>b_{m_0}$.  Since $n>n_1$, this $m_0$ is
greater than or equal to $n_0$.  By Lemma~\ref{3.8}, for each $m+1\geq m_0$,
$\xi_{m+1}$ is chosen so that $A(\nu,R_{m+1},T,b_{m},\ell_{m+1})/\base{\ell_{m+1}}{r}^2$  is
sufficiently small to ensure
\[
D(\expa{r^j\xi_m}:\base{b_m}{r}<j\leq
\base{b_{m+1}}{r})<(\epsilon_{m+1}/10)^4.
\]  
By Lemma~\ref{3.9}, for each $m$ greater than or equal to $m_0$,
\[
D(\expa{r^j\xi}:\base{b_m}{r}<j\leq
\base{b_{m+1}}{r})<\epsilon_m\leq\epsilon_{m_0}.
\]
Fix $m$ so that $\base{b_m}{r}\leq \base{n}{r}<\base{b_{m+1}}{r}$.  By a direct count,
\[
D(\expa{r^j\xi}:\base{b_{m_0}}{r}<
j\leq\base{b_m}{r})<\epsilon_{m_0}.
\]
By Lemma~\ref{3.3},
\[
D(\expa{r^j\xi}:0\leq j< \base{b_m}{r})<2\epsilon_{m_0}.
\]
And again by Lemma~\ref{3.3}, 
\[
D(\expa{r^j\xi}:0\leq j< \base{n}{r})<4\epsilon_{m_0}=f(n).
\]
Furthermore, since $\lim_{n\to\infty}f(n)=0$ we have 
$\lim_{n\to\infty} D(\expa{r^j\xi}:0\leq j<\base{n}{r})=0$.  Consequently, $\xi$ is normal base $r$.

Consider the base $s$.  During each stage $m$, the value of
$
D(\expa{s^j\xi_m}:\base{b_m}{s}<j\leq \base{b_{m+1}}{s})
$
is controlled from above and from below.
First, we discuss the lower bound on the discrepancy function for $\xi$ in base $s$.  By
construction, $\xi_{m+1}$ is obtained from $\xi_m$ by adding a rational number whose $s^{k_m}$-adic
expansion omits at least the digit $s^{k_m}-1$.  Further, the same digit $s^{k_m}-1$ in base
$s^{k_m}$ was omitted every previous stage (omitting $s^{k}-1$ in base $s^k$ precludes a length $k$
sequence of digits $s-1$ in base $s$).  Then, for any $n$ such that 
${\base{b_m}{s}}\leq n<{\base{b_{m+1}}{s}}$, 
\[
D(\expa{s^j\xi}:0\leq j< \base{n}{s})\geq
1/(2s^{k_m}).
\]
By construction, $k_m$ is defined so that $1/(2s^{k_m})>g({\base{b_m}{s}})\geq
g(n)$.  Hence,
\[
D(\expa{s^j\xi}:0\leq j< \base{n}{s})> g(n).
\]
Now, we treat the upper bound.  Let $m$ be a stage.  Let $m_0$ be the greatest stage less than or
equal to $m$ such that $k_{m_0}\neq k_{m_0-1}$.  By construction, $k_{m_0}$ and $\ell_{m_0}$ satisfy
the conditions of Lemma~\ref{3.14} with input $\epsilon$ equal to $1/(4s^{k_{m_0}})$.  Since
$\ell_m\geq\ell_{m_0}$, the same holds during stage $m$.  Consider the selection of $\nu$
during stage $m$.  By Lemma~\ref{3.17}, more than half of the eligible candidates satisfy the
inequality 
$A(\nu,R_m,T,b_{m-1},\ell_m)/\base{\ell_m}{\max(R_m)}^2<\delta$.  
By Lemma~\ref{3.14}, more than half the candidates satisfy
\[
D(\expa{s^j\nu}:1\leq j\leq \base{\ell_m}{s})<1/(4{s^k_{m_0}}).
\]
Consequently, $\xi_m$ will be defined so that
\[
D(F,(\expa{s^j\xi_m}:\base{b_{m-1}}{s}< j\leq \base{b_m}{s}))
\]
is less than $1/(4 s^k_{m_0} )$, where $F$ is as indicated in the construction.  By Lemma~\ref{3.1},  
\[
D(\expa{s^j\xi_m}:\base{b_{m-1}}{s}< j\leq
\base{b_m}{s})<3(4 s^{k_{m_0}})^{-1/2}.
\]
As already argued, $\lim_{m\to\infty}\epsilon_m=0$.  Similarly, the
values of $\kbar_m$ and the maximum element of $R_m$ become arbitrarily large as $m$ increases.  It
follows that $\lim_{m\to\infty}\ell_m=\infty$.  Since $g$ is a monotonically decreasing function and
$\lim_{n\to\infty}g(n)=0$, for every stage $m$ there will be a later stage $m_1$ such that 
$k_{m_1}>k_{m}$.  Thus, $\lim_{m\to\infty}
D(\expa{s^j\xi_m}:\base{b_{m-1}}{s}< j\leq \base{b_m}{s})=0$.
It follows from Lemma~\ref{3.9}, that $\lim_{m\to\infty}
D(\expa{s^j\xi}:\base{b_{m-1}}{s}< j\leq \base{b_m}{s})=0$, and
from Lemma~\ref{3.4} that $\lim_{N\to\infty} D(\expa{s^j\xi}:0\leq j< N)=0$.
Hence $\xi$ is normal to base $s$.  By Maxfield's Theorem, $\xi$ is normal to every base
multiplicatively dependent to $s$.  
Thus, $\xi$ is absolutely normal.

\subsection{Proof of Theorem~\ref{5}}

\begin{theoremC}
  Let $R$ be a set of bases closed under multiplicative dependence.  There are real numbers normal
  to every base from $R$ and not simply normal to any base in its complement.  Furthermore, such a real
  number can be obtained computably from $R$.
\end{theoremC}

Let $S$ denote the set of bases in the complement of $R$.  Fix an enumeration of $S$ such that every
element of $S$ appears infinitely often.  The case in which $2$ is an element of $S$ requires
special attention and we treat it separately.

\medskip\emph{The case $2\not\in S$.\ } Assume that $2$ is not an element of $S$.  Fix an
  enumeration of $S$ in which every element of $S$ appears infinitely often.
  We define a sequence $\xi_m$, $b_m$, $s_m$, $\epsilon_m$, $\ell_m$, $R_m$ and $c_m$.  $b_m$ is a positive
  integer, $\epsilon_m$ a positive rational number and $R_m$ a finite set of bases multiplicatively
  independent to $s_m$.  $\xi_m$ is an $s_m$-adic rational number of precision $\base{b_m}{s_m}$.
  The real $\xi$ will be an element of $[\xi_m,\xi_m+s_m^{-\base{b_m}{s_m}})$.  Stage $m+1$ is
  devoted to extending $\xi_{m}$ so that the discrepancy of the extension is below $\epsilon_{m+1}$
  for the bases in $R_{m+1}$ and so that the extension in base $s_{m+1}$ omits the digit
  $s_{m+1}-1$.  $\ell_m$ determines the length of the extension.  $c_m$ is a counter to track
  progress through the enumeration of $S$ with repetitions.

  \noindent{\em Initial stage.\ } Let $\xi_0=0$, $b_0=0$, $s_0$ be the least element of $S$,
  $\epsilon_0=1$, $\ell_0=0$, $R_0=\{r_0\}$ where $r_0$ is the least element of $R$ and $c_0=1$

\noindent{\em Stage $m+1$.\ }
Given 
$\xi_{m}$ of the form $\sum_{j=1}^{\base{b_{m}}{s^{k_{m}}}} v_j(s_m^{k_{m}})^{-j}$,
$b_{m}$, 
$s_{m}$,
$\epsilon_{m}$,
$\ell_m$,
$R_m$ and $c_m$.

(1) If $D(\{[1-1/s_m,1]\},(\expa{s_m^j\xi_m}:0\leq j<\base{b_m}{s_m}))<(1/4)(1/s_m)$, then let
$s_{m+1}=s_{m}$,
$\epsilon_{m+1}=\epsilon_{m}$,
$\ell_{m+1}=\ell_m$,
and $R_{m+1}=R_m$.  
\medskip

(2) Otherwise, let $c=c_m+1$.  Let $s$ be the $c$th element in the enumeration of $S$.  Let $r$ be
the least element of $R$ not in $R_m$.  Let $L$ be the least integer greater than $\max(c\, p(s_m,s)\log(\max(R_m)), \ell(R_m\cup\{r\},s,1,1/c))$.  If~$(1/c) \base{b_m}{\max(R_m)}\leq L+p(s_m,s)$ then let $s_{m+1}$ be $s_m$, let
$\epsilon_{m+1}$ be $\epsilon_m$, let $\ell_{m+1}$ be $\ell_m$, $R_{m+1}$ be $R_m$ and $c_{m+1}$ be
$c_m$.
\medskip

(3) Otherwise, let $s_{m+1}$ be $s$, $\epsilon_{m+1}$ be $1/c$, $\ell_{m+1}$ be $L$, $R_{m+1}$ be
$R_m\cup\{r\}$ and $c_{m+1}$ be $c$.  
\medskip

Let $a_{m+1}$ be minimal such that there is an $s_{m+1}$-adic subinterval of
$[\xi_{m},\xi_{m}+s_{m}^{-\base{b_{m}}{s_{m}}})$ with measure
$s_{m+1}^{-\base{a_{m+1}}{s_{m+1}}}$ and the leftmost such subinterval be
$[\eta_{m+1},\eta_{m+1}+s_{m+1}^{-\base{a_{m+1}}{s_{m+1}}})$.  Let $\stilde$
be $s_{m+1}-1$ if $s_{m+1}$ is odd and be $s_{m+1}-2$ otherwise.  Let $T$ and
$\delta$ be as determined in Lemma~\ref{3.8} for input $\epsilon=(\epsilon_{m+1}/10)^4$.  Let $\nu$
be in $[\eta_{m+1},\eta_{m+1}+s_{m+1}^{-\base{a_{m+1}}{s_{m+1}}})$ such that

\begin{itemize}
\item $\displaystyle{ 
    \nu=\eta_{m+1}+\sum_{j=1}^{\base{\ell_{m+1}}{s_{m+1}}}  w_js_{m+1}^{-(\base{a_{m+1}}{s_{m+1}}+j)}}$,
  for some $(w_1,\dots,w_{\base{\ell_{m+1}}{s_{m+1}}})$ in\\
  $\digits{\stilde}{\base{\ell_{m+1}}{s_{m+1}}}$

\item  $A(\nu,R_{m+1},T,b_{m},\ell_{m+1})/\base{\ell_{m+1}}{\max(R_{m+1})}^2<\delta$
\end{itemize}
We define $\xi_{{m+1}}$ to be $\nu$ and $b_{m+1}$ to be $a_{m+1}+\ell_{m+1}$.  This ends the
description of stage $m+1$.

We verify that the construction succeeds.
Let $m+1$ be a stage.  If clause (1) or (2) applies during stage $m+1$, let $m_0$ be the greatest
stage less than or equal to $m+1$ such that $c_{m_0}=c_{m_0+1}=\cdots=c_{m+1}$.  During stage $m_0$,
$\ell_{m_0}$ was chosen to satisfy the conditions to reach clause (3).  Note that since
$b_m>b_{m_0}$ these conditions apply to $b_m$ in place of $b_{m_0}$: $(1/c_{m+1})
\base{b_m}{\max(R_{m_0})}>\ell_{m+1}+p(s_{m_0-1},s_{m+1})$ and $\ell_{m+1}$ is the maximum of $c_{m+1}
p(s_{m_0-1},s_{m+1})$ and $\ell(R_{m+1},s_{m+1},1,1/c_{m+1})$.  If clause (3) applies during stage
$m+1$, then the analogous conditions hold by construction.
Then, stage $m+1$ determines the subinterval
$[\eta_{m+1},\eta_{m+1}+(s_{m+1})^{-\base{a_{m+1}}{s_{m+1}}})$ of the interval provided at the end
of stage $m$.  Following that, it selects $\nu$ and finishes the stage.  The existence of an
appropriate $\nu$ is ensured by Lemma~\ref{3.17} applied to the parameters of the construction, as
anticipated in the definition of the $\ell$ function.  It follows that $\xi$ is well defined as the
limit of the $\xi_{m}$.  Further, since $\ell$ takes only positive values, $b_m$ is an increasing
function of $m$.

We show that $c_m$ goes to infinity and $\epsilon_m=1/c_m$ goes to $0$.  Consider a stage $m+1$.
By construction, no element of
$(\expa{s_{m+1}^j\xi_{m+1}}:\base{a_{m+1}}{s_{m+1}}< j\leq \base{b_{m+1}}{s_{m+1}})$ is in
$[1-1/s_{m+1},1]$.  Further, during every subsequent stage $m_1+1$ with $c_{m_1+1}=c_{m+1}$, we have
$a_{m_1+1}=b_{m_1}$, so no element of $(\expa{s_{m+1}^j\xi_{m_1+1}}:\base{b_{m_1}}{s_{m+1}}< j\leq
\base{b_{m_1+1}}{s_{m+1}})$ is in $[1-1/s_{m+1},1]$.  By Lemma~\ref{3.5}, there will be a stage
$n+1$ after $m+1$ such that $c_{n+1}=c_{m+1}$ and
\[
D(\{[1-1/s_{m+1},1]\},(\expa{s_{m+1}^j\xi_n}:0\leq j<
\base{b_n}{s_{m+1}}))\geq (1/4)(1/s_{m+1}).
\] 
Thus, clauses (1) and (2) cannot maintain the value $c_{m+1}$ indefinitely.

Suppose that $s\in S$.  There will be infinitely many stages $m$  such that $s=s_m$.  By the above, there
will be infinitely many $m$ such that $s_m=s$ and
\[
D(\{[1-1/s_{m},1]\},(\expa{s_{m}^j\xi_m}:0\leq j<\base{b_m}{s_{m}}))\geq (1/4)(1/s_{m}).
\]
 Since $\xi\in[\xi_{m},\xi_{m}+s_{m}^{-\base{b_{m}}{s_{m}}})$, the same is true for $\xi$ in place of
$\xi_m$.  It follows that $\xi$ is not simply normal to base $s$.

Suppose that $r\in R$ and $\epsilon>0$.  For all sufficiently large stages, $r\in R_{m+1}$ and
$\epsilon_{m+1}<\epsilon$.  Consider a sufficiently large stage $m+1$.
$\xi_{m+1}$ was defined to be $\nu$, which was chosen so that
$A(\nu,R_{m+1},T,a_{m+1},\ell_{m+1})/\base{\ell_{m+1}}{\max(R_{m+1})}^2<\delta$.
By Definition~\ref{3.16},
\[A(\nu,R_{m+1},T,a_{m+1},\ell_{m+1}) =\displaystyle{\sum_{t\in T}\;\sum_{r\in R_{m+1}}\;
  \Bigl|\sum_{j=\base{a_{m+1}+1}{r}}^{\base{b_{m+1}}{r}} e(r^j t \nu)\Bigr|^2 }
\]
 and so
$\displaystyle{\base{\ell_{m+1}}{r}^{-2} \sum_{t\in T}\Bigl|\sum_{j=\base{a_{m+1}+1}{r}}^{\base{b_{m+1}}{r}} e(r^j t
 \nu)\Bigr|^2<\delta.}$ 
By choice of $T$ and $\delta$, Lemma~\ref{3.8} ensures that
\[
D(r^j\nu:\base{a_{m+1}}{r}<j\leq \base{b_{m+1}}{r})<(\epsilon_{m+1}/10)^4.
\]
By definition of $\xi$, $\xi \in[\nu,\nu+(s_{m+1}^{k_{m+1}})^{-\base{b_{m+1}}{s_{m+1}^{k_{m+1}}}})$.  By Lemma~\ref{3.9}, we
conclude that
\[
D(r^j\xi:\base{a_{m+1}}{r}< j\leq \base{b_{m+1}}{r})<\epsilon_{m+1}.
\]
By construction,  $\epsilon_{m+1} \ell_{m+1}$ is greater than $\log(r)\,p(s_m,s_{m+1})$.  By Lemma~\ref{3.3}
\[
D(r^j\xi:\base{b_{m}}{r}< j\leq \base{b_{m+1}}{r})<2\epsilon_{m+1}<2\epsilon.
\]
It follows that $\xi$ is normal to base $r$.

\medskip\emph{The case $2\in S$.\ }   Removing $2$ and
retaining all of its other powers in $S$ maintains the condition of multiplicative independence
between elements of $R$ and $S$.  A small alteration in our construction during the stages that
ensure that $\xi$ is not simply normal for base $4$ will also ensure that $\xi$ is not simply normal
for base $2$, by application of Lemma~\ref{3.13}:

We change clause (2) to require that $\ell_m$ be sufficiently large so that Lemma~\ref{3.13} applies
to conclude that more than half of the base $4$ sequences of length $\ell_m$ have simple discrepancy
greater than $1/8$ in base $2$.  This requirement is added to the others that determine $\ell_m$ in
the general construction.  Then, while the value $s_m=4$ is maintained, we choose $\nu$ from among
these sequences and so that the condition on the value of $A$ on $\nu$ from the general construction
is also satisfied.  Finally, clause (1) should be changed so that in addition to the existing
condition on discrepancy in base~$4$ there is another condition that  
the simple discrepancy in base $2$ is less than $1/16$.

Even with these changes, $\xi$ is well-defined.  Lemma~\ref{3.13} shows that more than half of the
sequences $\nu$ have simple discrepancy greater than $1/8$ in base $2$.  Lemma~\ref{3.17} shows that
at least half of them satisfy the condition on the value of $A$.  Thus, there is an appropriate
$\nu$ available.  Arguing as previously, $\xi$ is not simply normal to base~$2$.
\medskip

\noindent {\bf Acknowledgements.} Becher's research was supported by CONICET and Agencia Nacional de
Promoci\'{o}n Cient\'{i}fica y Tecnol\'{o}gica, Argentina.  Slaman’s research was partially supported
by the National Science Foundation, USA, under Grant No. DMS-1001551 and by the Simons Foundation.
This research was done while the authors participated in the Buenos Aires Semester in Computability,
Complexity and Randomness, 2013.

\bibliography{bases}

\end{document}